\documentclass[onefignum,onetabnum]{siamart190516}

\usepackage{lipsum}
\usepackage{amsfonts}
\usepackage{graphicx}
\usepackage{epstopdf}
\usepackage{xcolor}
\usepackage{tikz}
\usepackage{algorithmic}
\usepackage{subfloat}
\usepackage{subfig}

\ifpdf
  \DeclareGraphicsExtensions{.eps,.pdf,.png,.jpg}
\else
  \DeclareGraphicsExtensions{.eps}
\fi

\usepackage{hyperref}
\usepackage{cleveref}
\crefname{equation}{}{}
\crefname{figure}{Fig.}{Figs.}
\crefname{appendix}{}{}
\crefname{table}{Tab.}{Tabs.}
\Crefname{ALC@unique}{Line}{Lines} 

\newcommand{\R}{\mathbb{R}}

\newcommand{\uu}{{\bf u}}
\newcommand{\vv}{{\bf v}}

\def\ol{\overline}
\def\bxi{{\boldsymbol \xi}}
\def\etab{\boldsymbol{\eta}}

\newcommand{\Po}{\mathcal{P}}
\newcommand{\B}{\mathcal{B}}
\newcommand{\T}{\mathcal{T}}
\newcommand{\p}{\mathrm{p}}
\newcommand{\s}{\mathrm{s}}
\newcommand{\ww}{\mathcal{W}}
\newcommand{\n}{{\bf n}}

\newcommand{\gray}[1]{\textcolor{gray!75}{#1}}

\DeclareMathAlphabet\mathbfcal{OMS}{cmsy}{b}{n}
\newsiamremark{remark}{Remark}
\newsiamremark{hypothesis}{Hypothesis}
\crefname{hypothesis}{Hypothesis}{Hypotheses}
\newsiamthm{claim}{Claim}

\headers{Invariant domain preserving approximations of hyperbolic systems}{V. Carlier and F.Renac}

\title{Invariant domain preserving high-order spectral discontinuous approximations of hyperbolic systems}

\author{Valentin Carlier\thanks{\'{E}cole normale sup\'erieure, F-75230 Paris, France ({\tt vcarlier@clipper.ens.fr})}
\and Florent Renac\thanks{DAAA, ONERA, Universit\'e Paris Saclay F-92322 Ch\^atillon, France ({\tt florent.renac@onera.fr})}}

\usepackage{amsopn}

\def \LOCALFIGPATHRP {./FIGURES/}

\ifpdf
\hypersetup{
  pdftitle={Invariant domain preserving approximations of hyperbolic systems},
  pdfauthor={V. Carlier and F.Renac}
}
\fi

\begin{document}

\maketitle

\begin{abstract}
We propose a limiting procedure to preserve invariant domains with time explicit discrete high-order spectral discontinuous approximate solutions to hyperbolic systems of conservation laws. Provided the scheme is discretely conservative and satisfy geometric conservation laws at the discrete level, we derive a condition on the time step to guaranty that the cell-averaged approximate solution is a convex combination of states in the invariant domain. These states are then used to define local bounds which are then imposed to the full high-order approximate solution within the cell via an a posteriori scaling limiter. Numerical experiments are then presented with modal and nodal discontinuous Galerkin schemes confirm the robustness and stability enhancement of the present approach.
\end{abstract}



%
%
\section{Introduction}

Let $D\subset \R^d$ be an open domain with $d$ the space dimension. We are interested here in high-order numerical solutions to hyperbolic systems of conservations law
\begin{equation}
\label{eqn:HCL}
\left\{  \begin{array}{ll}
\partial_t \uu + \nabla \cdot {\bf f}(\uu)=0,& \text{ in } D \times (0, \infty),\\
\uu({\bf x},0)=\uu_0({\bf x}), & \text{ in } D,
\end{array}
\right. 
\end{equation}

\noindent where $\uu({\bf x}, t)$ represents the vector of conserved variables with values in the set of states $\Omega^a \subset \R^m$ which is assumed to be convex. The flux tensor ${\bf f}=({\bf f}_1,\dots,{\bf f}_d) : \Omega^a \ni \uu \mapsto {\bf f}(\uu) \in \R^{m \times d}$ is assumed to be smooth. 

Solutions to \cref{eqn:HCL} may develop discontinuities in finite time even if the initial data is smooth, therefore the equations are to be understood in the sense of distributions. Nevertheless, in this setting we lose uniqueness of the solution and \cref{eqn:HCL} must be supplemented with further admissibility criteria. We here focus on entropy inequalities on some twice differentiable strictly convex function $\eta : \Omega^a \rightarrow \R$ associated with a smooth entropy flux ${\bf q}: \Omega^a \rightarrow \R^{d} $ satisfying 
\begin{equation}
    \label{eqn:entropflux}
    \etab'(\uu)^\top {\bf f}_i'(\uu)={\bf q}_i'(\uu)^\top \quad \forall \uu \in \Omega^a, \quad 1\leq i\leq d.
\end{equation}

A weak solution to \cref{eqn:HCL} is called an entropy weak solution if for every entropy pair of \cref{eqn:HCL} we have
\begin{equation}
    \label{eqn:entropicsol}
    \frac{\partial \eta(\uu)}{ \partial t}+\nabla \cdot {\bf q}(\uu)\leq0,
\end{equation}

\noindent in the sense of distributions. Classical (smooth) solutions respect this condition with an equality, since one can apply the chain rule with \cref{eqn:entropflux}. The inequality for discontinuous solutions comes from a vanishing viscosity argument by adding a parabolic perturbation to \cref{eqn:HCL}, the regularizing effect allows to have a smooth unique solution in this case. Then, under some structural assumptions on \cref{eqn:HCL}, vanishing viscosity approximations converge to an entropy measure valued solution to \cref{eqn:HCL,eqn:entropicsol} \cite{Diperna1983ConvergenceOA}. This result is in particular based on the existence of convex invariant domains $\B \subset \Omega^a $ for \cref{eqn:HCL}: if $\uu$ is in $\B$, then it remains in $\B$ almost everywhere in $D\times(0,\infty)$ \cite{Chueh_elal_IDP_sys_77,Frid_idp_LF_01,Hoff_idp_85,serre_inv_reg_87}. This property generalizes the notion of maximum principle for scalar equations. Numerical methods keeping this property at the discrete level are called invariant domain preserving (IDP).

We are here interested in the approximation of \cref{eqn:HCL} using high-order discontinuous spectral methods (see, e.g., \cite{cockburn-shu89,gassner_13,fisher_carpenter_13,chan_skewDG_2019,Crean_etal_SBP_curved_2018} and references therein) where the solution to \cref{eqn:HCL} is sought under the form of discontinuous piecewise truncated series of analytic functions over a partition of the domain $D$. Such methods have been applied to a wide range of applications \cite{shu2016high,wang2013high}, and have the potential to achieve high-order accuracy efficiently on modern parallel architectures \cite{hutchinson2016efficiency,franco2020high}. Unfortunately these approximations suffer from spurious oscillations around discontinuities of the exact solution due to Gibbs phenomenon \cite{don1994numerical,gottlieb1997gibbs} that may cause the approximate solution to become locally nonphysical, leading to robustness issues. A large body of research has been proposed to address such issues with, e.g., solution and flux limiters \cite{KRIVODONOVA_etal_lim_DG_04,zhang_shu_10a,Guermond_etal_IDP_conv_lim_19}, entropy conservative subcell flux differencing \cite{fisher_carpenter_13,chan_skewDG_2019,Crean_etal_SBP_curved_2018} artificial viscosity \cite{GUERMOND_etal_EV_04,BarterDarmofal_AVDG_10}, shock-capturing terms \cite{Jaffre_etal_cv_DGFE_95,hiltebrand_mishra_14}.

We here focus on a posteriori limiters from \cite{zhang_shu_10a,zhang2010positivity} scaling the cellwise approximate solution around its cell average, thus allowing to preserve positivity of the solution (i.e., with $\B=\Omega^a$) and maximum principles in scalar problems, while preserving conservativity of the method. Under some strong assumptions on the mesh, it is indeed possible to derive a condition on the time step of the scheme so that the cell-averaged solution remains in the invariant domains on Cartesian and simplicial grids \cite{zhang_shu_10a,zhang2010positivity,Jiang_Lu_IDP_DG_18}, or on unstructured quadrangular straight-sided grids \cite{renac2020entropy}. We here extend this limiting technique to an IDP limiter for a broad class of spectral discontinuous methods with explicit time stepping on general unstructured meshes with possibly curved elements. Provided, the discretization method is conservative and satisfies geometric conservation laws \cite{thomas_lombard_GCL_79,kopriva_metric_id_06} at the discrete level, we propose a condition on the time step to guaranty that the cell-averaged approximate solution is a convex combination of states lying in the required invariant domains. These states are then used to define local bounds which are then imposed to the full high-order approximate solution within the cell via the scaling limiter. This strategy is closely related to convex limiting \cite{Guermond_etal_IDP_conv_lim_19} based on first-order IDP approximations defining local bounds and then forcing the high-order approximation to satisfy these bounds through flux limiting \cite{BORIS_Book_FCT_73,zalesak1979fully}. This approach has been applied to finite element approximations in \cite{Guermond_etal_IDP_conv_lim_18} and discontinuous Galerkin spectral element method in \cite{PAZNER_idg_DGSEM20211} among others. 

The objective of this paper is hence to derive a CFL condition on the time step and to propose an iterative algorithm for its evaluation. These are based first on the existence of a state in the invariant region which satisfies a trivial flux balance over each mesh cell boundaries. We call this state the pseudo-equilibrium state and then use tricks from \cite{perthame_shu_96} to expand the the cell-averaged approximate solution is a convex combination of states lying in the invariant domains. The CFL condition is also based on the existence of a quadrature rule to evaluate the cell-averaged solution that includes the traces of the numerical solution used to evaluate numerical fluxes in the scheme. This latter result generalized the work in \cite{zhang2012maximum} on triangular grids to general curved polyhedral elements.

The paper is organized as follows. In \cref{sec:RP} we introduce the notion of invariant domain and invariant domain preserving Riemann solver. In \cref{sec:CFL} we will state and prove our theorem on the existence of a CFL for high order schemes, present and discuss the limiting strategy. \cref{sec:schemes} will present various schemes that satisfy the hypothesis of our work and state the CFL precisely for those. Numerical experiments will be presented in \cref{sec:experiments}, and the conclusions follow in
\cref{sec:conclusions}.

%
%
\section{Approximate Riemann solvers}\label{sec:RP}

In this section we present some basic notions \cite{bouchut_04,hll_83} on Riemann problem, approximate Riemann solver (ARS), and convex invariant domain that will be used in the remainder of this work. Throughout this section, $\n$ in $\mathbb{R}^d$ is a given unit vector.

\subsection{Riemann problem and invariant domains}

Let two states $\uu_L$ and $\uu_R$ in $\Omega^a$, it is convenient for the present analysis to consider the Riemann problem in the direction $\n$: 
\begin{equation}\label{eqn:RP}
 \partial_t \uu + \partial_x{\bf f}(\uu) \cdot \n =0, \text{ in } \R \times (0, \infty), \quad \uu(x,0) = \left\{  \begin{array}{ll}
\uu_L, & \text{if } x<0, \\
\uu_R, & \text{if } x>0.
\end{array}
\right.
\end{equation}

\noindent We will integrate \cref{eqn:RP} over the space time slab $[-\frac{h}{2},\frac{h}{2}]\times[0,\Delta t]$ with $h>0$ and $\Delta t>0$ the space and time steps. We suppose here that all the Riemann problems we consider have a self-similar entropy weak solution $\ww(\frac{x}{t}; \uu_L, \uu_R, \n)$. Let introduce the self-similar variable $\xi=\frac{x}{\Delta t}$ and assume that there exist $\sigma_L$, $\sigma_R$ such that: $\ww(\xi;\uu_L,\uu_R,\n)=\uu_L$ for $\xi<\sigma_L$ and $\ww(\xi;\uu_L,\uu_R,\n)=\uu_R$ for $\xi>\sigma_R$. We then define the maximum wave speed in \cref{eqn:RP} by
\begin{equation}\label{eqn:maxwavespeed}
|\lambda|(\uu_L, \uu_R, \n)=\max{}(|\sigma_L|,|\sigma_R|),
\end{equation}

\noindent and for $\frac{\Delta t}{h}|\lambda|(\uu_L,\uu_R,\n)\leq \frac{1}{2}$, we define the average over the Riemann fan \cite{Guermond_etal_IDP_conv_lim_19,hll_83}
\begin{equation}\label{eqn:Riemavg}
    {\bar \uu}(\uu_L, \uu_R,  \n,\Delta t) :=\frac{1}{h}\int_{-\frac{h}{2}}^{\frac{h}{2}}\ww\Big(\frac{x}{\Delta t}; \uu_L, \uu_R, \n\Big) dx = \frac{\uu_L+ \uu_R}{2}-\Delta t\big({\bf f}(\uu_R) -{\bf f}(\uu_L)\big)\cdot \n.
\end{equation}

%

Finally, we will use the definition of invariant domain from \cite{Guermond_etal_IDP_conv_lim_19}: a convex set $\B \subset \Omega^a$ is an invariant domain \cref{eqn:HCL} if for all $\uu_L$ and $\uu_R$ in $\B$, we have
\begin{equation*}
{\bar \uu}\big(\uu_L, \uu_R,  \n,\tfrac{\Delta t}{h}\big) \in \B \quad \forall \; \frac{\Delta t}{h}|\lambda|(\uu_L,\uu_R,\n)\leq \frac{1}{2}.
\end{equation*}

\subsection{Two-point numerical fluxes and approximate Riemann solvers}
The discretization of \cref{eqn:HCL} will rely on two-point numerical fluxes \cite{Lax_hyp_sys_cons_laws_73, hll_83} for the approximation of ${\bf f}\cdot{\bf n}$ and we assume them to be consistent and conservative:
\begin{equation}\label{eq:conserv_consist_flux}
 {\bf h}(\uu,\uu,\n)={\bf f}(\uu)\cdot\n, \quad {\bf h}(\uu_L,\uu_R,\n)=-{\bf h}(\uu_R,\uu_L,-\n) \quad \forall \uu,\uu_L, \uu_R \in \Omega^a,
\end{equation}

\noindent and Lipschitz continuous. We also define the notion of IDP two-point flux in the following definition.

\begin{definition}\label{def:invdom2ptflx}
A two-point flux is said to be invariant domain preserving (IDP) for $\B$ an invariant domain if we have
\begin{equation*}
  \uu-\frac{\Delta t}{h} \big( {\bf h}(\uu,\uu_R, \n)- {\bf h}(\uu_L,\uu, \n) \big) \in \B \quad \forall \; \uu_L, \uu, \uu_R \in \B,
\end{equation*}

\noindent under the half CFL condition
\begin{equation*}
    \frac{\Delta t}{h} \max (|\lambda|(\uu_L,\uu,\n),|\lambda|(\uu,\uu_R,\n) ) \leq \frac{1}{2}.
\end{equation*}

\end{definition}

We now introduce the notion of ARS and IDP ARS, which will be used to derive IDP two-point fluxes.

\begin{definition}[Approximate Riemann solver]
    An ARS is a self-similar function $\ww^a(\frac{x}{t}; \uu_L, \uu_R, \n)$, used to approximate the solution $\ww(\frac{x}{t}; \uu_L, \uu_R, \n)$ of the Riemann problem \cref{eqn:RP}, that is consistent with the integral form of \cref{eqn:HCL} \cite{hll_83,godunov_59}: for any $\frac{\Delta t}{h} |\lambda|(\uu_L,\uu_R, \n) \leq \frac{1}{2}$, we have
\begin{equation}
\label{eqn:consistRiemsolver}
        \frac{1}{h}\int_{-\frac{h}{2}}^{\frac{h}{2}}\ww^a\Big(\frac{x}{\Delta t};\uu_L,\uu_R,\n\Big)dx=\frac{\uu_L+\uu_R}{2}-\Delta t\big({\bf f}(\uu_R) -{\bf f}(\uu_L) \big)\cdot \n.
    \end{equation}
\end{definition}

Using consistency in \cref{eq:conserv_consist_flux}, and setting $\xi=\frac{x}{\Delta t}$, one defines a two-point flux from an ARS as 
\begin{subequations}\label{eqn:approxRiemFlux}
\begin{align}
{\bf h}_{\ww^a}(\uu_L,\uu_R,\n)&= {\bf f}(\uu_L)\cdot \n-\int_{-\lambda}^0 \big( \ww^a(\xi; \uu_L, \uu_R, \n)-\uu_L \big) d\xi \label{eqn:approxRiemFlux-a}\\
&= {\bf f}(\uu_R) \cdot \n+\int_0^{\lambda} \big(\ww^a(\xi; \uu_L, \uu_R,\n)-\uu_R \big) d\xi, \label{eqn:approxRiemFlux-b}
\end{align}
\end{subequations}

\noindent where $\lambda=|\lambda|(\uu_L,\uu_R,\n)$. Both definitions are equivalent due to \cref{eqn:consistRiemsolver}. We can now define the notion of IDP ARS.

\begin{definition}\label{def:IDP_ARS}
An ARS $\ww^a(\xi; \uu_L, \uu_R,\n)$ is IDP for $\B$ an invariant domain if we have
\begin{equation*}
 \frac{1}{\lambda}\int_{-\lambda}^0 \ww^a(\xi; \uu_L, \uu_R,\n) d\xi \in \B, \quad 
 \frac{1}{\lambda}\int_0^{\lambda} \ww^a(\xi; \uu_L, \uu_R, \n) d\xi \in \B \quad \forall \uu_L,\uu_R\in\B.
\end{equation*}
\end{definition}

As a consequence $\frac{1}{2 \lambda}\int_{-\lambda}^{\lambda} \ww^a(\xi, \uu_L, \uu_R,\n) d\xi $ is also in $\B$. We have the following results linking IDP ARS and two-point numerical flux. 

\begin{lemma}[Interface invariant domain preservation \cite{bouchut_04}]
\label{lem:interfIRP}
The ARS $\ww^a$ is IDP for $\B$ iff. for all $\uu_L, \uu_R$ in $\B$, and $\frac{\Delta t}{h} |\lambda|(\uu_L, \uu_R, \n) \leq 1$, we have
\begin{subequations}
\begin{align}
    \uu_L &- \frac{\Delta t}{h}\big({\bf h}_{\ww^a}(\uu_L,\uu_R,\n)-{\bf f}(\uu_L) \cdot \n \big) \in \B, \label{eqn:idparv1} \\
    \uu_R &- \frac{\Delta t}{h}\big({\bf f}(\uu_R) \cdot \n-{\bf h}_{\ww^a}(\uu_L,\uu_R,\n) \big) \in \B. \label{eqn:idparv2}
\end{align}
\end{subequations}
\end{lemma} 

\begin{proof}
Let consider \cref{eqn:idparv1}, a similar argument holds for \cref{eqn:idparv2}. From \cref{eqn:approxRiemFlux-a}, we have 
\begin{equation*}
\uu_L-\frac{\Delta t}{h}\big({\bf h}_{\ww^a}(\uu_L,\uu_R,\n)-{\bf f}(\uu_L) \cdot \n \big)= \big(1-\frac{\Delta t}{h} \lambda\big) \uu_L +  \frac{\Delta t}{h} \lambda \frac{1}{\lambda}\int_{-\lambda}^0 \big( \ww^a(\xi, \uu_L, \uu_R,\n) d\xi,
\end{equation*}

\noindent and since $0<\frac{\Delta t}{h} \lambda \leq 1$ this is a convex combination of states in $\B$ and therefore is in $\B$. Conversely, taking $\frac{\Delta t}{h}\lambda=1$ the above integrals are $\B$ iff. \cref{eqn:idparv1} and \cref{eqn:idparv2} hold and we conclude from \cref{def:IDP_ARS}.
\end{proof}

This allows us to state the following result.

\begin{lemma}
\label{lem:stabriemannflux}
Let ${\bf h}_{\ww^a}$ and ${\bf h}_{{\mathcal{W}}^b}$ be two-point fluxes from two different ARS that are IDP for $\B$. Then, we have
\begin{equation}
\label{eqn:stab2flux}
    \uu-\frac{\Delta t}{h} \big( {\bf h_{\ww^a}}(\uu,\uu_R, \n)- {\bf h_{{\mathcal{W}}^b}}(\uu_L,\uu, \n) \big) \in \B \quad \forall \uu_L, \uu, \uu_R \in \B,
\end{equation}

\noindent under the half CFL condition
\begin{equation*}
    \frac{\Delta t}{h} \max\big(|\lambda|(\uu_L,\uu,\n),|\lambda|(\uu,\uu_R,\n) \big) \leq \frac{1}{2}.
\end{equation*}
\end{lemma}

\begin{proof}
We rewrite \cref{eqn:stab2flux} as
\begin{align*}
    \uu-\frac{\Delta t}{h} \big( {\bf h_{\ww^a}}(\uu,\uu_R, \n)- {\bf h_{{\mathcal{W}}^b}}(\uu_L,\uu, \n) \big) &= \frac{1}{2} \Big( \uu-2\frac{\Delta t}{h} \big( {\bf h_{\ww^a}}(\uu,\uu_R, \n)- {\bf f}(\uu) \cdot \n \big) \Big)\\
    &+ \frac{1}{2} \Big( \uu-2\frac{\Delta t}{h} \big( {\bf f}(\uu) \cdot \n - {\bf h_{{\mathcal{W}}^b}}(\uu,\uu_R, \n) \big) \Big),
\end{align*}
and apply \cref{lem:interfIRP} to both terms of the right-hand side with $2\frac{\Delta t}{h} \lambda \leq 1$.
\end{proof}

The previous lemma with ${\mathcal{W}}^b=\ww^a$ also proves that the three-point scheme built from an IDP ARS is also IDP \cite{bouchut_04}.

%
%
\section{Invariant domain preserving limiter}\label{sec:CFL}

We here state and prove our main results on the existence of an explicit condition on the time step to ensure that the cell-averaged solution from a high-order spectral discontinuous scheme is IDP.  In \cref{sec:cell-av-scheme} we clarify the schemes we are considering in this work. Our results are based on the existence of a pseudo-equilibrium state allowing a balance of the numerical fluxes at faces of each element which is introduced in \cref{sec:pseudo_eq_state} where we prove its existence. The main result givin the condition on the time step is given in \cref{sec:Dt_IDP}. A limiting strategy based on convex bounds is described in \cref{sec:convex_limiter}, while \cref{sec:evaluation_ustar} introduces a fast algorithm to evaluate the time step.

\subsection{Cell-averaged fully discrete scheme}\label{sec:cell-av-scheme}

We now describe the main properties of the numerical methods we are considering in this work. We consider here discretely conservative high-order approximations of \cref{eqn:HCL}. Without loss of generality, we use an explicit forward Euler discretization in time. High-order time integration is then performed using strong-stability preserving Runge-Kutta methods \cite{gottlieb2001strong} that are convex combinations of explicit first-order schemes in time and thus keep their stability properties. For the spatial discretization, the approximate solution $\uu_h({\bf x},t)$ is defined locally over each element $\kappa$ of the partition $\T_h$ of the domain $D$ in a local function cell space ${\cal V}_h^p(\kappa)$. By $\uu_h^{(n+1)}(\cdot)=\uu_h(\cdot,t^{(n+1)})$ we denote the solution at time $t^{(n+1)}=t^{(n)}+\Delta t^{(n)}$ with $t^{(0)}=0$ and $\Delta t^{(n)}>0$ the time step. The approximate solution is assumed to satisfy the following relation for the cell-averaged solution $\langle\uu_h\rangle_\kappa$:
\begin{equation}
    \label{eqn:meaneqngen}
    \langle\uu_h^{(n+1)}\rangle_\kappa =\langle\uu_h^{(n)}\rangle_\kappa - \Delta t^{(n)}\sum_{k=1}^{N_f}s_k^\kappa{\bf h}\big(\uu_h^-({\bf x}_k^\kappa,t^{(n)}),\uu_h^+({\bf x}_k^\kappa,t^{(n)}),\n_k^\kappa \big) \quad \forall \kappa\in\T_h,
\end{equation}

\noindent where the ${\bf x}_k^\kappa$ are some points on the faces $f$ in $\partial\kappa$ and $\uu_h^\pm({\bf x}_k^\kappa,t^{(n)})= \lim_{\epsilon \rightarrow 0^+}\uu_h({\bf x}_k^\kappa \pm \epsilon \n_k^\kappa,t^{(n)})$ denote evaluations of the traces of the solutions at ${\bf x}_k^\kappa$ (see \cref{fig:stencil_2D}). The $s_k^\kappa>0$ are local contributions to $\frac{|f|}{|\kappa|}$ with $|f|$ and $|\kappa|$ approximations of the face surface and element volume, and we introduce
\begin{equation}\label{eqn:defbordkappa}
    {\cal S}^\kappa := \sum\limits_{k=1}^{N_f}s_k^\kappa.
\end{equation}

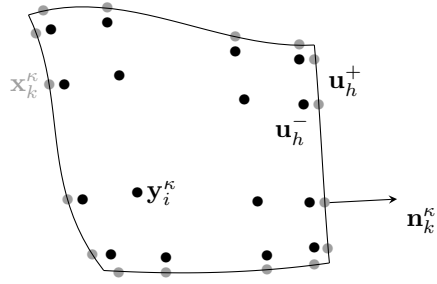
\begin{figure}[ht]
\begin{center}
\begin{tikzpicture}
[declare function={la(\x) = (\x+1/sqrt(5))/(-1+1/sqrt(5)) * (\x-1/sqrt(5))/(-1-1/sqrt(5)) * (\x-1)/(-2);
                   lb(\x) = (\x+1)/(-1/sqrt(5)+1) * (\x-1/sqrt(5))/(-2/sqrt(5)) * (\x-1)/(-1/sqrt(5)-1);
                   lc(\x) = (\x+1)/(1/sqrt(5)+1) * (\x+1/sqrt(5))/(2/sqrt(5)) * (\x-1)/(1/sqrt(5)-1);
                   ld(\x) = (\x+1)/(2) * (\x+1/sqrt(5))/(1+1/sqrt(5)) * (\x-1/sqrt(5))/(1-1/sqrt(5));
                   }]
\def\xad{-1.00};  \def\yad{3.40}; \def\xbd{0.05}; \def\ybd{3.49}; \def\xcd{1.75}; \def\ycd{3.11}; \def\xdd{2.80}; \def\ydd{3.00};
\def\xac{-0.723}; \def\yac{2.46}; \def\xbc{0.21}; \def\ybc{2.59}; \def\xcc{1.87}; \def\ycc{2.27}; \def\xdc{2.86}; \def\ydc{2.20};
\def\xab{-0.48};  \def\yab{0.94}; \def\xbb{0.45}; \def\ybb{1.03}; \def\xcb{2.05}; \def\ycb{0.91}; \def\xdb{2.94}; \def\ydb{0.90};
\def\xaa{0.00};   \def\yaa{0.00}; \def\xba{0.83}; \def\yba{-0.03}; \def\xca{2.17}; \def\yca{0.01}; \def\xda{3.00}; \def\yda{0.10};
\draw [>=stealth,->] (2.94,0.90) -- (3.9,0.95) ;
\draw (2.86,2.20) node[below left]  {${\bf u}_h^-$};
\draw (2.86,2.20) node[above right] {${\bf u}_h^+$};
\draw (3.9 ,0.95) node[below right] {$\n_k^\kappa$};
\draw (\xac,\yac) node[left] {$\gray{{\bf x}_k^\kappa}$};
\draw (\xbb,\ybb) node[right]  {${\bf y}_i^\kappa$};
\draw (-0.80,3.45) node {$\gray\bullet$}; \draw (\xbd,\ybd) node {$\gray\bullet$}; \draw (\xcd,\ycd) node {$\gray\bullet$}; \draw (2.60,\ydd) node {$\gray\bullet$};
\draw (-0.9,3.15) node {$\gray\bullet$};  \draw (\xdd,2.80) node {$\gray\bullet$};
\draw (\xac,\yac) node {$\gray\bullet$};  \draw (\xdc,\ydc) node {$\gray\bullet$};
\draw (\xab,\yab) node {$\gray\bullet$};  \draw (\xdb,\ydb) node {$\gray\bullet$};
\draw (-0.15,0.20) node {$\gray\bullet$};  \draw (2.98,0.29) node {$\gray\bullet$};
\draw (0.20,-0.03) node {$\gray\bullet$}; \draw (\xba,\yba) node {$\gray\bullet$}; \draw (\xca,\yca) node {$\gray\bullet$}; \draw (2.80,0.07) node {$\gray\bullet$};
\draw (-0.70,3.20) node {$\bullet$}; \draw (\xbd,3.29) node {$\bullet$}; \draw (\xcd,2.91) node {$\bullet$}; \draw (2.60,2.80) node {$\bullet$};
\draw (-0.523,\yac) node {$\bullet$}; \draw (\xbc,\ybc) node {$\bullet$}; \draw (\xcc,\ycc) node {$\bullet$}; \draw (2.66,\ydc) node {$\bullet$};
\draw (-0.28,\yab) node {$\bullet$}; \draw (\xbb,\ybb) node {$\bullet$}; \draw (\xcb,\ycb) node {$\bullet$}; \draw (2.74,\ydb) node {$\bullet$};
\draw (0.10,0.20) node {$\bullet$}; \draw (\xba,0.17) node {$\bullet$}; \draw (\xca,0.19) node {$\bullet$}; \draw (2.80,0.30) node {$\bullet$};
\draw [domain=-1:1] plot ({la(\x)*\xaa+lb(\x)*\xba+lc(\x)*\xca+ld(\x)*\xda}, {la(\x)*\yaa+lb(\x)*\yba+lc(\x)*\yca+ld(\x)*\yda});
\draw [domain=-1:1] plot ({la(\x)*\xad+lb(\x)*\xbd+lc(\x)*\xcd+ld(\x)*\xdd}, {la(\x)*\yad+lb(\x)*\ybd+lc(\x)*\ycd+ld(\x)*\ydd});
\draw [domain=-1:1] plot ({la(\x)*\xda+lb(\x)*\xdb+lc(\x)*\xdc+ld(\x)*\xdd}, {la(\x)*\yda+lb(\x)*\ydb+lc(\x)*\ydc+ld(\x)*\ydd});
\draw [domain=-1:1] plot ({la(\x)*\xaa+lb(\x)*\xab+lc(\x)*\xac+ld(\x)*\xad}, {la(\x)*\yaa+lb(\x)*\yab+lc(\x)*\yac+ld(\x)*\yad});
%
\end{tikzpicture}
\caption{Notations for $d=2$ on a quadrangle: definitions of the unit outward normal vector $\n_k^\kappa$, element quadrature nodes ${\bf y}_i^\kappa$ (black bullets), surface quadrature node ${\bf x}_k^\kappa$ (gray bullets), and inner and outer traces ${\bf u}_h^\pm$ at ${\bf x}_k^\kappa$.}
\label{fig:stencil_2D}
\end{center}
\end{figure}

The geometrical quantities depend on the numerical method under consideration and examples will be given in \cref{sec:schemes}. By ${\bf h}$ we denote a consistent, conservative \cref{eq:conserv_consist_flux}, and IDP (see \cref{def:invdom2ptflx}) two-point flux. The cell-averaged solution $\langle \uu_h^{(n)} \rangle_\kappa$ is supposed to be evaluated through a suitable quadrature rule, that includes $N_v$ volume quadrature points ${\bf y}_i^\kappa$ in $\kappa$ together with the $N_f$ surface points ${\bf x}_i^\kappa$ on $\partial\kappa$ (see \cref{fig:stencil_2D}) introduced in \cref{eqn:meaneqngen}:
\begin{equation}\label{eqn:surfquad}
\langle \uu_h^{(n)} \rangle_{\kappa}=\sum\limits_{i=1}^{N_v} \nu_i^\kappa \uu_h({\bf y}_i^\kappa,t^{(n)}) + \sum\limits_{i=1}^{N_f} \beta_i^\kappa \uu_h^-({\bf x}_i^\kappa,t^{(n)}),
\end{equation}

\noindent where the weights $\nu_{1\leq i\leq N_v}^\kappa \geq 0$ and $\beta_{1\leq i\leq N_f}^\kappa>0$ are assumed to satisfy
\begin{equation}
    \label{eqn:surfquadsum}
    \sum\limits_{i=1}^{N_v}\nu_i^\kappa+\sum\limits_{i=1}^{N_f}\beta_i^\kappa=1.
\end{equation}

\Cref{lem:existquad} and \cref{th:quad_modal} give a theoretical basis on existence of such a quadrature on polygonal or polyhedral mesh elements. An explicit example is also given in \cite{zhang2012maximum} in the case of triangles and in \cref{sec:schemes}.

\begin{lemma}
\label{lem:existquad}
Let $\kappa$ be a compact subset of $\R^d$ and $\Po_\kappa$ a finite dimensional subspace of ${\cal C}^0(\kappa,\R)$, that contains the constant function $f\equiv 1_\kappa$. Suppose that there exists a quadrature $(\varpi_i^\kappa,{\bf y}_i^\kappa)_{1\leq i\leq N_v}$ with positives weights $\varpi_i^\kappa>0$ and points ${\bf y}_i^\kappa$ in $\kappa$ that integrates exactly products of functions in $\Po_\kappa$:
\begin{equation}
    \label{eqn:exactint}
    \int_\kappa f({\bf y}) g({\bf y}) dV =\sum\limits_{i=1}^{N_v} \varpi_i^\kappa f({\bf y}_i^\kappa) g({\bf y}_i^\kappa) \quad \forall f,g \in \Po_\kappa.
\end{equation}

Then, for given points $({\bf x}_i^\kappa)_{1\leq i\leq N_f}$ in $\kappa$ there exist nonnegative $(\nu_i^\kappa)_{1\leq i\leq N_v}$ and positive $(\beta_i^\kappa)_{1\leq i\leq N_{f}}$ coefficients such that 
\begin{equation}\label{eqn:exactint_on_Pk}
 \langle f \rangle_{\kappa}=\sum\limits_{i=1}^{N_v} \nu_i^\kappa f({\bf y}_i^\kappa) + \sum\limits_{i=1}^{N_f} \beta_i^\kappa f({\bf x}_i^\kappa) \quad \forall f \in \Po_\kappa.
\end{equation}
\end{lemma}

\begin{proof}
It is known that $(f,g) \mapsto \int_{\kappa} f({\bf y})g({\bf y})dV$ defines a scalar product on ${\cal C}^0(\kappa,\R)$ and therefore on $\Po_\kappa$. Then using the Riesz representation theorem, every linear form $\varphi: \Po_\kappa \rightarrow \R$ can be represented using this scalar product: there exists $f_{\varphi}\in \Po_\kappa$ such that for every $g \in \Po_\kappa$, $\varphi(g)=\int_\kappa f_{\varphi}({\bf y}) g({\bf y}) dV$, then let $\alpha^\varphi_i=\varpi_i^\kappa f_{\varphi}({\bf y}_i^\kappa)$, we obtain for every $g \in \Po_\kappa$, $\varphi(g)=\sum_{i=1}^{N_v} \alpha_i^\varphi g({\bf y}_i^\kappa)$. Now, since $f \mapsto \sum_{i=1}^{N_{f}} s_k^\kappa f({\bf x}_i^\kappa)$ defines a linear form on $\Po_\kappa$, for any $s_k^\kappa>0$, it can be represented this way: there exist $(\alpha_i^\kappa)_{1\leq i\leq N_v}$ such that:
\begin{equation}\label{eqn:defalpha2}
    \sum\limits_{i=1}^{N_{f}} s_i^\kappa f({\bf x}_i^\kappa)=\sum\limits_{i=1}^{N_v} \alpha_i^\kappa f({\bf y}_i^\kappa)\quad\forall f \in \Po_\kappa.
\end{equation}

As the constant function is in $\Po_\kappa$, we have for $f$ in $\Po_\kappa$, $\langle f \rangle_\kappa = \sum_{i=1}^{N_v} \varpi_i^\kappa f({\bf y}_i^\kappa)$, so for $\varepsilon_\kappa>0$  
\begin{align*}
\langle f \rangle_\kappa = \sum\limits_{i=1}^{N_v} \varpi_i^\kappa f({\bf y}_i^\kappa) 
 &= \sum\limits_{i=1}^{N_v} \varpi_i^\kappa f({\bf y}_i^\kappa) - \varepsilon_\kappa \sum\limits_{i=1}^{N_{f}} s_i^\kappa f({\bf x}_i^\kappa) + \varepsilon_\kappa \sum\limits_{i=1}^{N_{f}} s_i^\kappa f({\bf x}_i^\kappa)\\
&= \sum\limits_{i=1}^{N_v} \varpi_i^\kappa f({\bf y}_i^\kappa) - \varepsilon_\kappa \sum\limits_{i=1}^{N_v} \alpha_i^\kappa f({\bf y}_i^\kappa) + \varepsilon_\kappa \sum\limits_{i=1}^{N_{f}} s_i^\kappa f({\bf x}_i^\kappa)\\
&= \sum\limits_{i=1}^{N_v} (\varpi_i^\kappa - \varepsilon_\kappa \alpha_i^\kappa) f({\bf y}_i^\kappa) + \varepsilon_\kappa \sum\limits_{i=1}^{N_{f}} s_i^\kappa f({\bf x}_i^\kappa).
\end{align*}

Since the $\varpi_i^\kappa$ are positive, for $\varepsilon_\kappa=\min_{\{i : \alpha_i^\kappa>0\}}(\frac{\varpi_i^\kappa}{\alpha_i^\kappa})>0$, the $(\varpi_i^\kappa - \varepsilon_\kappa \alpha_i^\kappa)$ are nonnegative, then \cref{eqn:exactint_on_Pk} holds with 
\begin{equation}\label{eqn:defalpha}
 \nu_i^\kappa=\varpi_i^\kappa - \varepsilon_\kappa \alpha_i^\kappa \geq0 \quad \forall 1\leq i\leq N_v, \quad \beta_i^\kappa=\varepsilon_\kappa s_i^\kappa>0 \quad \forall 1\leq i\leq N_{f}.
\end{equation}
\end{proof}

\begin{remark}
Note that $\kappa$ is usually a polyhedron and $\Po_\kappa$ a polynomial space, we can subdivide $\kappa$ into simplices and since there are quadrature rules integrating exactly arbitrary order polynomials on simplices, the previous lemma can be applied. Likewise, the existence of the quadrature in \cref{lem:existquad} is required only for modal methods or when the DOFs are not defined at the ${\bf x}_k^\kappa$ in \cref{eqn:meaneqngen}, so the present framework also holds for non polynomial nodal approximations with DOFs at the faces as in \cite{Crean_etal_SBP_curved_2018}.
\end{remark}

The following corollary allows to explicitely define the quadrature \cref{eqn:surfquad} for modal polynomial based methods.

\begin{corollary}\label{th:quad_modal}
Given a basis $(\phi_j)_{1\leq j\leq N_p}$ of $\Po_\kappa$ which is orthonormal with respect to the inner product (i.e., $\int_\kappa\phi_i({\bf x})\phi_j({\bf x})dV=\delta_{ij}$ where $\delta_{ij}$ is the Kronecker symbol), then the $\alpha_i^\kappa$ in \cref{eqn:defalpha} read 
\begin{equation}
\label{eqn:alphamodal}
    \alpha_i^\kappa=\varpi_i^\kappa\sum\limits_{j=1}^{N_p} \sum\limits_{k=1}^{N_{f}} s_k^\kappa \phi_j({\bf x}_k^\kappa) \phi_j ({\bf y}_i^\kappa).
\end{equation}
\end{corollary}

\begin{proof}
We know that there exists $g \in \Po_\kappa$ such that (see proof of \cref{lem:existquad})
\begin{equation}
\label{eqn:surfrepbasis}
    \sum\limits_{i=1}^{N_{f}} s_i^\kappa f({\bf x}_i^\kappa) = \sum\limits_{i=1}^{N_v} \varpi_i^\kappa f({\bf y}_i^\kappa) g({\bf y}_i^\kappa) \quad \forall f \in \Po_\kappa.
\end{equation}

Expanding $g\equiv\sum\limits_{k=1}^{N_p} g_k \phi_k$ in the orthonormal basis and using \cref{eqn:surfrepbasis} with $f\equiv\phi_j$ we get 
\begin{equation*}
    \sum\limits_{i=1}^{N_{f}} s_i^\kappa \phi_j({\bf x}_i^\kappa)=\sum\limits_{i=1}^{N_v} \varpi_i^\kappa \phi_j({\bf y}_i^\kappa) \sum\limits_{k=1}^{N_p} g_k \phi_k ({\bf y}_i^\kappa) =g_j,
\end{equation*}

\noindent for all $1\leq j\leq N_p$, by orthonormality of the basis. Substituting $g$ in \cref{eqn:surfrepbasis} by its expansion in the basis gives
\begin{equation*}
     \sum\limits_{i=1}^{N_{f}} s_i^\kappa f({\bf x}_i^\kappa)=\sum\limits_{i=1}^{N_v} \varpi_i^\kappa f({\bf y}_i^\kappa) \sum\limits_{j=1}^{N_p} \sum\limits_{k=1}^{N_{f}} s_k^\kappa \phi_j({\bf x}_k^\kappa) \phi_j ({\bf y}_i^\kappa) \quad\forall f \in \Po_\kappa,
\end{equation*}

\noindent and we conclude by comparing this result with \cref{eqn:defalpha2}.
%
\end{proof}

Finally, the scheme is assumed to preserve uniform states in the following sense:
\begin{equation}
\label{eqn:closesurface}
\sum\limits_{k=1}^{N_f} s_k^\kappa \n_k^\kappa=0,
\end{equation}

\noindent which is a discrete version of  $\tfrac{1}{|\kappa|}\oint_{\partial\kappa} \n dS=0$ for a closed contour. Relation \cref{eqn:closesurface} is closely related to the discrete geometric conservation laws \cite{thomas_lombard_GCL_79} and is required for the numerical scheme to preserve free-stream states \cite{kopriva_metric_id_06}. In \cref{sec:experiments} we will present schemes that satisfies assumptions \cref{eqn:meaneqngen}, \cref{eqn:surfquad} and \cref{eqn:closesurface}.

\subsection{The pseudo-equilibrium state}\label{sec:pseudo_eq_state}

We first introduce the Rusanov flux \cite{Rusanov1961} : 
\begin{equation}\label{eq:Rusanov_flux}
 {\bf h}_\lambda(\uu_L, \uu_R,\n)=\frac{{\bf f}(\uu_L)\cdot \n+{\bf f}(\uu_R)\cdot \n}{2}-\frac{\lambda}{2}(\uu_R-\uu_L),
\end{equation}

\noindent for $\lambda \geq |\lambda|(\uu_L,\uu_R,\n)$ defined in \cref{eqn:maxwavespeed}. Note that the Rusanov flux is derived from the following ARS:
\begin{equation*}
{\cal W}_\lambda(\xi,\uu_L,\uu_R,\n)=\left\{
\begin{array}{ll}
     \uu_L, & \xi < -\lambda,  \\
     \frac{\uu_L+\uu_R}{2}-\frac{1}{2 \lambda}({\bf f}(\uu_R)\cdot \n-{\bf f}(\uu_L)\cdot \n), & -\lambda < \xi < \lambda, \\
     \uu_R, & \lambda < \xi,  \\
\end{array}
\right.
\end{equation*}

\noindent so from \cref{eqn:Riemavg} we know that it is IDP, see also \cite{Frid_idp_LF_01}. We now state a result that will allow us to rewrite \cref{eqn:meaneqngen} with updates of three-point schemes.

\begin{lemma}[pseudo-equilibrium state]
\label{lem:bar}
Suppose that the numerical scheme satisfies \cref{eqn:meaneqngen}, \cref{eqn:surfquad} and \cref{eqn:closesurface}. Let $\B$ be a invariant domain and suppose that the internal traces $\uu_h^-({\bf x}_k^\kappa,t^{(n)})_{1\leq k\leq N_f}$ are in $\B$, then there exists $\uu^\star_\kappa=\uu^\star_\kappa(t^{(n)})$ in $\B$ and ${\lambda}_\kappa^\star={\lambda}_\kappa^\star(t^{(n)})>0$ finite such that  
 \begin{equation}
 \label{eqn:defubar}
    \sum_{k=1}^{N_f}s_k^\kappa{\bf h}_{ \lambda_\kappa^\star}\big(\uu^\star_\kappa,\uu_h^-({\bf x}_k^\kappa,t^{(n)}),\n_k^\kappa\big)=0, 
\end{equation}

\noindent with ${\bf h}_{\lambda_\kappa^\star}$ defined in \cref{eq:Rusanov_flux} with $\lambda=\lambda_\kappa^\star$ where
\begin{equation}
\label{eqn:deflambar}
    \lambda^\star_\kappa \geq \max_{1\leq k\leq N_f} \big(|\lambda|( \uu^\star_\kappa, \uu_h^-({\bf x}_k^\kappa,t^{(n)}),\n_k^\kappa)\big),
\end{equation}
 
\noindent and the pseudo-equilibrium state is defined by
\begin{equation}
    \label{eqn:expustar1}
    \uu^\star_\kappa=\sum_{k=1}^{N_f} \Tilde{\gamma}_k^\kappa \bigg(\uu_h^-({\bf x}_k^\kappa,t^{(n)})-\frac{{\bf f}\big(\uu_h^-({\bf x}_k^\kappa,t^{(n)})\big)\cdot \n_k^\kappa}{ \lambda^\star_\kappa}\bigg), \quad \Tilde{\gamma}_k^\kappa:=\frac{s_k^\kappa}{{\cal S}^\kappa}.
\end{equation}
\end{lemma}

\begin{proof}
For the sake of clarity, we remove the time dependance of $\uu_h$ since all evaluations are done at $t^{(n)}$ and write $\uu_h^-({\bf x}_k^\kappa)$ for $\uu_h^-({\bf x}_k^\kappa,t^{(n)})$. We first remark that from \cref{eqn:defbordkappa,eqn:expustar1}, we have
\begin{equation}
    \label{eqn:sumtildgamma}
    \sum_{k=1}^{N_f}\Tilde{\gamma}_k^\kappa=1.
\end{equation}

We introduce the following two sequences:
\begin{equation}
\label{eqn:sequences}
\left\{  \begin{array}{l}
\uu^\star_0=\sum\limits_{k=1}^{N_f} \Tilde{\gamma}_k^\kappa \uu_h^-({\bf x}_k^\kappa), \quad
\lambda_0 = \max\limits_{1\leq k\leq N_f}\big( |\lambda|( \uu^\star_0 , \uu_h^-({\bf x}_k^\kappa),\n_k^\kappa)\big)\Big), 
\\
\lambda_{p+1}=\max\Big(\lambda_p, \frac{1}{d(\uu^\star_p,\partial B)} + \max\limits_{1\leq k\leq N_f}\big( |\lambda|( \uu^\star_p , \uu_h^-({\bf x}_k^\kappa),\n_k^\kappa)\big)\Big), \quad p\geq0,\\
\uu^\star_{p+1}=\displaystyle\sum_{k=1}^{N_f} \Tilde{\gamma}_k^\kappa \bigg(\frac{\uu_h^-({\bf x}_k^\kappa)+\uu^\star_p}{2}-\frac{{\bf f}\big(\uu_h^-({\bf x}_k^\kappa)\big)\cdot \n_k^\kappa- {\bf f}(\uu^\star_p)\cdot \n_k^\kappa}{2 \lambda_{p+1}}\bigg), \quad p\geq0,
\end{array}
\right. 
\end{equation}

\noindent where $d(\uu, \partial \B)=\inf_{\vv \in \partial \B}\|\uu-\vv\|$ is the distance from $\uu$ to the boundary of $\B$. We will show that both sequences converge and the limits satisfy \cref{eqn:defubar} and \cref{eqn:deflambar}. We first remark that $(\lambda_p)$ is non decreasing and therefore converges to some $\lambda^\star_\kappa$ in $\R \cup \{+ \infty \}$ and that for all $p\geq 0$, $\uu^\star_{p}$ is in $\B$ since $\lambda_{p+1}\geq |\lambda|(\uu_h^-({\bf x}_k^\kappa), \uu^\star_p,\n_k^\kappa)$ so we can use \cref{eqn:Riemavg} with $\Delta t=\frac{1}{2 \lambda_{p+1}}$. Using successively \cref{eqn:closesurface} and the definition of $\Tilde{\gamma}_k^\kappa$ in \cref{eqn:expustar1}, then \cref{eqn:sumtildgamma}, we have
\begin{align*}
 \uu^\star_{p+1}&=\sum\limits_{k=1}^{N_f} \Tilde{\gamma}_k^\kappa \bigg(\frac{\uu_h^-({\bf x}_k^\kappa)+\uu^\star_p}{2}-\frac{{\bf f}(\uu_h^-({\bf x}_k^\kappa))\cdot \n_k^\kappa}{2 \lambda_{p+1}}\bigg) \\
 &=\frac{\uu^\star_p}{2}+\frac{1}{2}\sum\limits_{k=1}^{N_f} \Tilde{\gamma}_k^\kappa \uu_h^-({\bf x}_k^\kappa)-\frac{1}{2\lambda_{p+1}}\sum\limits_{k=1}^{N_f} \Tilde{\gamma}_k^\kappa {\bf f}(\uu_h^-({\bf x}_k^\kappa))\cdot \n_k^\kappa.
\end{align*}

Then the first step of the sequence \cref{eqn:sequences} for $\uu^\star_\kappa$ reads
\begin{equation*}
\uu^\star_{1}=\sum\limits_{k=1}^{N_f} \Tilde{\gamma}_k^\kappa \uu_h^-({\bf x}_k^\kappa)-\frac{1}{2\lambda_{1}}\sum\limits_{k=1}^{N_f} \Tilde{\gamma}_k^\kappa {\bf f}(\uu_h^-({\bf x}_k^\kappa))\cdot \n_k^\kappa,
\end{equation*}

\noindent so applying the recurrence relation $p$ more times, we obtain
\begin{equation}
\label{eqn:expustar}
    \uu^\star_{p+1}=\sum\limits_{k=1}^{N_f} \Tilde{\gamma}_k^\kappa \uu_h^-({\bf x}_k^\kappa)- \sum\limits_{i=1}^{p+1} \frac{1}{2^i \lambda_{(p+2-i)}} \times \sum\limits_{k=1}^{N_f} \Tilde{\gamma}_k^\kappa {\bf f}(\uu_h^-({\bf x}_k^\kappa))\cdot \n_k^\kappa.
\end{equation}

Let show that $\sum_{i=1}^p \frac{1}{2^i \lambda_{(p+1-i)}}$ above converges to $\frac{1}{\lambda^\star_\kappa}$. Assume $\lambda^\star_\kappa$ is finite and let $\epsilon>0$, since $\lambda_p$ converges to $\lambda^\star_\kappa$, there exists $p_0$ such that $|\frac{1}{\lambda_p}-\frac{1}{\lambda^\star_\kappa}|<\epsilon$ for all $p>p_0$. Then for $p>p_0$, we set
\begin{align*}
\sum\limits_{i=1}^p \frac{1}{2^i \lambda_{(p+1-i)}} - \frac{1}{\lambda^\star_\kappa}&=\sum\limits_{k=1}^p \frac{1}{2^{(p+1-k)} \lambda_k} -\frac{1}{\lambda^\star_\kappa} \\
&=\sum\limits_{k=1}^{p_0} \frac{1}{2^{(p+1-k)} \lambda_k}+\sum\limits_{k=p_0+1}^{p} \frac{1}{2^{(p+1-k)} \lambda_k} - \frac{1}{\lambda^\star_\kappa} \\
&=\frac{1}{2^{p+1}}\sum\limits_{k=1}^{p_0} \frac{2^k}{\lambda_k} +\!\!\! \sum\limits_{k=p_0+1}^{p} \frac{2^k}{2^{p+1}}\left(\frac{1}{\lambda_k}\!-\!\frac{1}{\lambda^\star_\kappa}\right) - \frac{1}{\lambda^\star_\kappa}\bigg(1 \!-\!\!\! \sum\limits_{k=p_0+1}^{p} \frac{2^k}{2^{p+1}}\bigg)
\end{align*}

\noindent and using a triangle inequality we obtain 
\begin{align*}
\left|\sum\limits_{i=1}^p \frac{1}{2^i \lambda_{(p+1-i)}} - \frac{1}{\lambda^\star_\kappa} \right|
\leq& \frac{1}{2^{(p+1)}}\sum\limits_{k=1}^{p_0} \frac{2^k}{ \lambda_k} +\!\! \sum\limits_{k=p_0+1}^{p} \frac{2^k}{2^{p+1}}\left|\frac{1}{\lambda_k}-\frac{1}{\lambda^\star_\kappa}\right| \\ 
 &+ \frac{1}{\lambda^\star_\kappa}\bigg(1-\!\! \sum\limits_{k=p_0+1}^{p} \frac{2^k}{2^{p+1}}\bigg) \\
\leq& \frac{1}{2^{(p+1)}}\sum\limits_{k=1}^{p_0} \frac{2^k}{ \lambda_k}+\sum\limits_{i=1}^{p-p_0} \frac{1}{2^i}\epsilon + \frac{1}{\lambda^\star_\kappa}\bigg(1-\sum\limits_{i=1}^{p-p_0}\frac{1}{2^i}\bigg) \\
=& \frac{1}{2^{p+1}}\sum\limits_{k=1}^{p_0} \frac{2^k}{\lambda_k} + \left(1-\frac{1}{2^{p-p_0}}\right)\epsilon + \frac{1}{2^{p-p_0}}\frac{1}{\lambda^\star_\kappa}
\end{align*}

\noindent which clearly converges to $\epsilon$ when $p\rightarrow\infty$, proving our statement and \cref{eqn:expustar1}. 

Let now prove that $\lambda^\star_\kappa$ is finite by contradiction. Suppose that $\lambda_p \rightarrow + \infty$, then $\uu^\star_p \rightarrow \sum_{k=1}^{N_f} \Tilde{\gamma}_k^\kappa \uu_h^-({\bf x}_k^\kappa)$, but the application
\begin{equation*}
 \uu \mapsto \max_{1\leq k\leq N_f}( |\lambda|( \uu, \uu_h^-({\bf x}_k^\kappa),\n_k^\kappa))+\frac{1}{d(\uu,\partial B)}
\end{equation*}

\noindent is continuous in the interior of $\B$ and is hence locally bounded around $\sum_{k=1}^{N_f} \Tilde{\gamma}_k^\kappa \uu_h^-({\bf x}_k^\kappa)$, implying that $\lambda_p$ is bounded and $\lambda^\star_\kappa$ is finite which is a contradiction. By \cref{eqn:sequences}, we obviously have $\lambda_{p+1} \geq \max_{1\leq k\leq N_f}\big( |\lambda|( \uu^\star_p, \uu_h^-({\bf x}_k^\kappa),\n_k^\kappa)\big)$ and passing the inequality to the limit we obtain \cref{eqn:deflambar}.

Let now prove that $\uu^\star_\kappa$ is in $\B$, since for all $p \geq 0$, $\uu^\star_p$ is in $\B$, we already know that $\uu^\star_\kappa$ is in the closure ${\bar \B}$. Now by contradiction, assuming that $\uu^\star_\kappa$ is not in $\B$, we necessarily have $\uu^\star_\kappa$ in $\partial\B$, so $d(\uu^\star_p,\partial\B)\rightarrow 0$ inducing $\lambda_p \rightarrow +\infty$ by \cref{eqn:sequences} which a contradiction. It remains to prove \cref{eqn:defubar}. Using \cref{eqn:closesurface}, we add $\frac{1}{\lambda^\star_\kappa}{\bf f}(\uu^\star_\kappa)\cdot\sum_{k=1}^{N_f} \Tilde{\gamma}_k^\kappa\n_k^\kappa=0$ to \cref{eqn:expustar1} and get
\begin{equation*}
    \uu^\star_\kappa=\sum_{k=1}^{N_f} \Tilde{\gamma}_k^\kappa \bigg(\uu_h^-({\bf x}_k^\kappa)-\frac{{\bf f}(\uu_h^-({\bf x}_k^\kappa))\cdot \n_k^\kappa+{\bf f}(\uu^\star_\kappa)\cdot \n_k^\kappa}{\lambda^\star_\kappa}\bigg).
\end{equation*}

Moving $\uu^\star_\kappa$ to the right-hand side and using \cref{eqn:sumtildgamma} we obtain
\begin{equation*}
    \sum_{k=1}^{N_f} \Tilde{\gamma}_k^\kappa \bigg(\uu_h^-({\bf x}_k^\kappa)-\uu^\star_\kappa - \frac{{\bf f}(\uu_h^-({\bf x}_k^\kappa))\cdot \n_k^\kappa+{\bf f}(\uu^\star_\kappa)\cdot \n_k^\kappa}{ \lambda^\star_\kappa}\bigg)=0,
\end{equation*}

\noindent and multiplying the above quantity by $-\frac{\lambda^\star_\kappa}{2}\sum_{i=1}^{N_f} s_i^\kappa=-\frac{\lambda^\star_\kappa}{2}{\cal S}^\kappa$ we finally get
\begin{equation*}
    \sum_{k=1}^{N_f} s_k^\kappa \bigg( \frac{{\bf f}(\uu_h^-({\bf x}_k^\kappa))\cdot \n_k^\kappa+{\bf f}(\uu^\star_\kappa)\cdot \n_k^\kappa}{2}-\lambda^\star_\kappa \frac{(\uu_h^-({\bf x}_k^\kappa)-\uu^\star_\kappa)}{2}\bigg)=0,
\end{equation*}

\noindent which is exactly \cref{eqn:defubar} from the definition of the Rusanov flux \cref{eq:Rusanov_flux}.
\end{proof}

\subsection{Invariant domain preserving schemes}\label{sec:Dt_IDP}

Using \cref{lem:bar} we now state and prove the main result of this work in the theorem below.

\begin{theorem}[Time step condition]
\label{thm:CFLHO}
Assume that the numerical scheme satisfies \cref{eqn:meaneqngen}, \cref{eqn:surfquad} and \cref{eqn:closesurface}  and assume that $u_h^\pm({\bf x}_k^\kappa,t^{(n)})_{1\leq k\leq N_f}$ and $u_h({\bf y}_i^\kappa,t^{(n)})_{1\leq i\leq N_v}$ are in $\B$, then under the following condition on the time step 
    \begin{equation}
    \label{eqn:CFL}
    \Delta t^{(n)} \max\limits_{\kappa \in \T_h} \max_{1\leq k\leq N_f} \frac{s_k^\kappa}{\beta_k^\kappa} \max\Big(\lambda^\star_\kappa, |\lambda|\big(\uu_h^-({\bf x}_k^\kappa,t^{(n)}),\uu_h^+({\bf x}_k^\kappa,t^{(n)}),\n_k^\kappa\big)\Big)\leq \frac{1}{2},
\end{equation}

\noindent where $\lambda^\star_\kappa$ is defined by \cref{eqn:deflambar}, $\langle u_h^{(n+1)} \rangle_{\kappa}$ is also in $\B$.
\end{theorem}

\begin{proof}
Once again we remove the time dependance of $\uu_h$ for the sake of clarity, except when explicitly needed in the evaluation of $\langle\uu_h^{(n+1)}\rangle_\kappa$. Using \cref{lem:bar} we add the trivial quantity $\Delta t^{(n)}\times$ \cref{eqn:defubar} to \cref{eqn:meaneqngen} and use \cref{eqn:surfquad} to get 
 \begin{align}
 \langle\uu_h^{(n+1)}\rangle_\kappa =& \langle\uu_h^{(n)}\rangle_\kappa - \Delta t^{(n)}\sum_{k=1}^{N_f}s_k^\kappa \big( {\bf h}(\uu_h^-({\bf x}_k^\kappa),\uu_h^+({\bf x}_k^\kappa),\n_k) - {\bf h}_{ \lambda^\star_\kappa}(\uu^\star_\kappa,\uu_h^-({\bf x}_k^\kappa),\n_k^\kappa) \big) \nonumber\\
 =& \sum_{i=1}^{N_v} \nu_i \uu_h({\bf y}_i^\kappa) + \sum_{k=1}^{N_f} \Big(\beta_k^\kappa \uu_h^-({\bf x}_k^\kappa) \nonumber\\
 &- \Delta t^{(n)} s_k^\kappa \big( {\bf h}(\uu_h^-({\bf x}_k^\kappa),\uu_h^+({\bf x}_k^\kappa),\n_k) - {\bf h}_{ \lambda^\star_\kappa}(\uu^\star_\kappa,\uu_h^-({\bf x}_k^\kappa),\n_k^\kappa )  \big) \Big) \nonumber\\
=&\sum_{i=1}^{N_v} \nu_i \uu_h({\bf y}_i^\kappa) + \sum_{k=1}^{N_f} \beta_k^\kappa {\cal U}_k^{\kappa,n}, \label{eq:convex_comb_un+1}
\end{align}
 
\noindent where, from \cref{lem:stabriemannflux} and the condition \cref{eqn:CFL}, the updates 
\begin{equation}\label{eq:def_alpha_k}
    {\cal U}_k^{\kappa,n} := \uu_h^-({\bf x}_k^\kappa) - \frac{\Delta t^{(n)}s_k^\kappa}{\beta_k^\kappa} \Big( {\bf h}\big(\uu_h^-({\bf x}_k^\kappa),\uu_h^+({\bf x}_k^\kappa),\n_k^\kappa\big) - {\bf h}_{\lambda^\star_\kappa}\big(\uu^\star_\kappa,\uu_h^-({\bf x}_k^\kappa),\n_k^\kappa\big)  \Big),
\end{equation} 

\noindent are in $\B$. Then by \cref{eqn:surfquadsum} $\langle\uu_h^{(n+1)}\rangle_\kappa$ is a convex combination of quantities in $\B$, which concludes the proof.
\end{proof}

\begin{remark}
Since the set of states $\Omega^a$ is in general a convex invariant domain, \cref{thm:CFLHO} can then be applied to $\B=\Omega^a$ to ensure robustness of the scheme.
\end{remark}

\begin{remark}
In the case where we are using the quadrature on the volume defined by \cref{lem:existquad}, from \cref{eqn:alphamodal,eqn:defalpha} and the definition $\varepsilon_\kappa=\min_{\{i: \alpha_i^\kappa>0\}}(\frac{\varpi_i^\kappa}{\alpha_i^\kappa})$, we have
\begin{equation}\label{eq:def_s_over_beta}
 \frac{s_k^\kappa}{\beta_k^\kappa}=\frac{1}{\varepsilon_\kappa}=\max_{\{i: \alpha_i^\kappa>0\}}(\frac{\alpha_i^\kappa}{\varpi_i^\kappa})=\max_{1\leq i\leq N_v}\sum_{j=1}^{N_p} \sum_{l=1}^{N_{f}} s_l^\kappa \phi_j({\bf x}_l^\kappa) \phi_j ({\bf y}_i^\kappa) \quad \forall 1\leq k \leq N_f,
\end{equation}

\noindent and the CFL condition \cref{eqn:CFL} now reads
\begin{equation}\label{eqn:CFLmodal}
  \Delta t^{(n)} \max\limits_{\kappa\in\T_h} \frac{1}{\varepsilon_\kappa} \max\Big(\lambda^\star_\kappa, 
	|\lambda|\big(\uu_h^-({\bf x}_k^\kappa,t^{(n)}),\uu_h^+({\bf x}_k^\kappa,t^{(n)}),\n_k^\kappa\big)\Big)\leq \frac{1}{2}.
\end{equation}

\end{remark}

%
\subsection{Limiting strategy}\label{sec:convex_limiter}

Convex limiting enforces the numerical solution to preserve invariant domains \cite{Guermond_etal_IDP_conv_lim_19} through quasiconcave constraints \cite{avriel1972r}. Let recall that a function $\psi : \B \rightarrow \R $ is quasiconcave iff. for every family of convex coefficients $(\lambda_i)\geq0$, with $\sum \lambda_i=1$, we have $\psi(\sum_{i} \lambda_i \uu_i) \geq \min_{i}\psi(\uu_i)$ for all $\uu_i$ in $\B$. From \cref{thm:CFLHO} we see that for any quasiconcave function $\psi$ we have 
\begin{equation}
\label{eqn:qcbound}
\psi(\langle \uu_h^{(n+1)} \rangle_{\kappa})\geq m_\kappa^{\psi}:= \min\Big(\psi\big(\uu_h({\bf y}_{i}^\kappa,t^{(n)})\big)_{1\leq i\leq N_v}, \psi\big({\cal U}_k^{\kappa,n}\big)_{1\leq k\leq N_f}\Big),
\end{equation}

\noindent where the updates ${\cal U}_k^{\kappa,n}$ are defined in \cref{eq:def_alpha_k}. We now limit the solution around its cell-average $\langle \uu_h^{(n+1)} \rangle_{\kappa}$ so that it satisfies the same bounds and we rely on scaling limiters introduced in \cite{zhang2010positivity,zhang_shu_10a} to enforce the bounds from quasiconcave functions to points where $\uu_h$ needs to be evaluated. The limited solution is thus defined as 
\begin{equation}
\label{eqn:IRPlim}
    \tilde{\uu}_{h}^{(n+1)} \equiv (1-\theta_\kappa)\uu_{h}^{(n+1)}+\theta_\kappa\langle \uu_h^{(n+1)} \rangle_{\kappa},
\end{equation}
\noindent where
\begin{equation*}
    \theta_\kappa = \!\! \min\limits_{{\bf z}\in ({\bf y}_{1\leq i\leq N_v}^\kappa) \cup ({\bf x}_{1\leq i\leq N_f}^\kappa)} \!\! \max\big\{0\leq t\leq 1: \psi\big((1-t)\uu_{h}({\bf z},t^{(n+1)})+t\langle \uu_h^{(n+1)}\rangle_{\kappa}\big)\geq m_{\kappa}^\psi\big\}.
\end{equation*}

This strategy may be applied to a finite family $(\psi_i)_{1\leq i\leq n_c}$ of $n_c$ quasiconcave functions by using the minimum value $\theta_\kappa=\min\{\theta_\kappa(\psi_i): 1\leq i\leq n_c\}$. The limiter \cref{eqn:IRPlim} is then applied locally to each cell $\kappa$ and preserves high-order accuracy of the scheme in smooth domains \cite{zhang2010positivity}. The cell-average is not modified, $\langle\tilde\uu_h^{(n+1)}\rangle_{\kappa}=\langle \uu_h^{(n+1)} \rangle_{\kappa}$, and cellwise discrete conservation \cref{eqn:meaneqngen} still holds.

\subsection{Practical evaluation of \texorpdfstring{$\uu^\star_\kappa$}{u*}}\label{sec:evaluation_ustar}
The evaluation of $m_\kappa^\psi$ in \cref{eqn:qcbound} requires to first evaluate both $\lambda^\star_\kappa$ and $\uu^\star_\kappa$ as limits of the sequences in \cref{eqn:sequences} which might be cumbersome. Here we propose another strategy where we only need to check the convergence of $(\lambda_p)$ to ensure that $(\uu^\star_p)$ has also converged to $\uu^\star_\kappa$ which is in $\B$. Let introduce the sequences
\begin{equation}\label{eqn:modifseq}
\left\{ \!\!\!
\begin{array}{l}
\lambda_0 = \frac{1}{\vartheta}\max\limits_{1\leq k\leq N_f} \! |\lambda|\big(\uu_h^-({\bf x}_k^\kappa,t^{(n)}) ,\uu_h^+({\bf x}_k^\kappa,t^{(n)}),\n_k^\kappa\big),
\uu^\star_0=\!\sum\limits_{k=1}^{N_f} \Tilde{\gamma}_k^\kappa \uu_h^-({\bf x}_k^\kappa,t^{(n)}), \\
\lambda_{p+1}=\max\Big(\lambda_p, \frac{1}{\vartheta}\max\limits_{1\leq k\leq N_f}\big( |\lambda|( \uu^\star_p , \uu_h^-({\bf x}_k^\kappa,t^{(n)}),\n_k^\kappa)\big)\Big), \quad p\geq0, \\
\uu^\star_{p+1}=\sum\limits_{k=1}^{N_f} \Tilde{\gamma}_k^\kappa \Big(\uu_h^-({\bf x}_k^\kappa,t^{(n)})-\frac{{\bf f}\big(\uu_h^-({\bf x}_k^\kappa,t^{(n)})\big)\cdot \n_k^\kappa}{\lambda_{p+1}} \Big), \quad p\geq 0,
\end{array}
\right.
\end{equation}

\noindent where 
\begin{align*}
  \vartheta &=\max\Big\{0\leq t\leq1: \uu^\star_0 - \frac{t}{\tilde\lambda_1}\sum\limits_{k=1}^{N_f}\Tilde{\gamma}_k^\kappa{\bf f}\big(\uu_h^-({\bf x}_k^\kappa,t^{(n)})\big)\cdot \n_k^\kappa\in\B\Big\}, \\
 \tilde\lambda_1 &= \max\limits_{1\leq k\leq N_f}\Big( |\lambda|\big(\uu_h^-({\bf x}_k^\kappa,t^{(n)}) ,\uu_h^+({\bf x}_k^\kappa,t^{(n)}),\n_k^\kappa\big), |\lambda|\big( \uu^\star_0 , \uu_h^-({\bf x}_k^\kappa,t^{(n)}),\n_k^\kappa\big)\Big).
\end{align*}

Since $\uu^\star_0 \in \B$, we have $0<\vartheta\leq1$ and then $\uu^\star_1 \in \B$ by definition of $\vartheta$. Now $\lambda_p$ is increasing and $\uu^\star_{p+1}\in[\uu^\star_0, \uu^\star_p]\subset[\uu^\star_0, \uu^\star_1]$ for all $p\geq 1$ which is enough to prove convergence of both sequences and ensure that $\uu^\star_p \in \B$ and since $[\uu^\star_0, \uu^\star_1]$ is closed the limit is also in $\B$ (no need to add the distance term as in \cref{eqn:sequences}). Obviously, the limits satisfy \cref{eqn:defubar}, \cref{eqn:deflambar} and \cref{eqn:expustar1}. But this time, if $\lambda_{p+1}=\lambda_p$ for $p\geq1$, then $\uu^\star_{p+1}=\uu^\star_{p}$ and from this point both sequences are stationary. Now, the evaluation of $\uu^\star_p$ is really fast and we only need to evaluate $\max_{1\leq k\leq N_f}( |\lambda|( \uu^\star_p , \uu_h^-({\bf x}_k^\kappa,t^{(n)}),\n_k^\kappa))$.

It is possible to use a local $\lambda^{\star k}_\kappa$ at each vertex ${\bf x}_\kappa^k$ in $\partial\kappa$ in the above algorithm to lower the artificial dissipation of the Rusanov flux \cref{eq:Rusanov_flux} and thus avoid the induced restriction on the time step as well as to reduce over-diffusion of the updates ${\cal U}_k^{\kappa,n}$ in \cref{eq:def_alpha_k}. We thus look for $\lambda^{\star k}_\kappa$, $1\leq k\leq N_f$, and $\uu^\star_\kappa$ satisfying 
\begin{equation}
 \label{eqn:locubar}
    \sum_{k=1}^{N_f} s_k^\kappa{\bf h}_{\lambda^{\star k}_\kappa}\big(\uu^\star_\kappa,\uu_h^-({\bf x}_k^\kappa,t^{(n)}),\n_k^\kappa\big)=0, 
\end{equation}

\noindent and
\begin{equation}
\label{eqn:loclambar}
    \lambda^{\star k}_\kappa \geq |\lambda|\big( \uu^\star_\kappa, \uu_h^-({\bf x}_k^\kappa,t^{(n)}),\n_k^\kappa\big).
 \end{equation}
 
We therefore introduce new sequences as
\begin{equation}\label{eqn:modifseq_loc}
\left\{  \begin{array}{ll}
\lambda_0^k = \frac{1}{\vartheta}|\lambda|\big( \uu_h^-({\bf x}_k^\kappa,t^{(n)}),\uu_h^+({\bf x}_k^\kappa,t^{(n)}),\n_k^\kappa\big),\;
\uu^\star_0=\sum\limits_{k=1}^{N_f} \Tilde{\gamma}_k^\kappa \uu_h^-({\bf x}_k^\kappa,t^{(n)}),  \\
\lambda_{p+1}^k=\max\Big(\lambda_p^k, \frac{1}{\vartheta}|\lambda|\big( \uu^\star_p, \uu_h^-({\bf x}_k^\kappa,t^{(n)}),\n_k^\kappa\big)\Big), \; p\geq 0,\\
\uu^\star_{p+1}=\sum\limits_{k=1}^{N_f} \frac{\Tilde{\gamma}_k^\kappa \lambda_{p+1}^k}{\sum_{i=1}^{N_f} \Tilde{\gamma}_i \lambda_{p+1}^i} \Big( \uu_h^-({\bf x}_k^\kappa,t^{(n)})-\frac{{\bf f}\big(\uu_h^-({\bf x}_k^\kappa,t^{(n)})\big)\cdot \n_k^\kappa}{\lambda_{p+1}^k} \Big), \quad p\geq0,
\end{array}
\right. 
\end{equation}

\noindent where the index $k$ ranges from $1$ to $N_f$ and
 \begin{align*}
 \vartheta &= \max\Big\{0\leq t\leq1: \sum\limits_{k=1}^{N_f} \frac{\Tilde{\gamma}_k^\kappa \lambda_{1}^k}{\sum_{i=1}^{N_f} \Tilde{\gamma}_i \lambda_{1}^i} \Big( \uu_h^-({\bf x}_k^\kappa,t^{(n)}) - t\frac{{\bf f}\big(\uu_h^-({\bf x}_k^\kappa,t^{(n)})\big)\cdot \n_k^\kappa}{\tilde\lambda_{1}^k} \Big) \in\B\Big\}, \\
 \tilde\lambda_{1}^k &= \max\Big( |\lambda|\big(\uu_h^-({\bf x}_k^\kappa,t^{(n)}) ,\uu_h^+({\bf x}_k^\kappa,t^{(n)}),\n_k^\kappa\big), |\lambda|\big( \uu^\star_0 , \uu_h^-({\bf x}_k^\kappa,t^{(n)}),\n_k^\kappa\big)\Big).
\end{align*}

Now all the sequences $(\lambda^k_p)_p$ are non-decreasing and will converge. Also for $p \geq 1$, we compute 
\begin{equation*}
    \Big(\sum_{i=1}^{N_f}\Tilde{\gamma}_i \lambda^i_{p+1} \Big) \uu^\star_{p+1}-\Big(\sum_{i=1}^{N_f}\Tilde{\gamma}_i \lambda^i_{p}\Big)\uu^\star_p = \sum_{k=1}^{N_f} \Tilde{\gamma}_k^\kappa (\lambda_{p+1}^k-\lambda_{p}^k) \uu_h^-({\bf x}_k^\kappa,t^{(n)}),
\end{equation*}

\noindent so $\uu^\star_{p+1}$ can be recast as a convex combination
\begin{equation*}
\uu^\star_{p+1}=\frac{\sum_{i=1}^{N_f}\Tilde{\gamma}_i \lambda^i_{p}}{\sum_{i=1}^{N_f}\Tilde{\gamma}_i \lambda^i_{p+1}}\uu^\star_p+ \sum_{k=1}^{N_f} \frac{\Tilde{\gamma}_k^\kappa (\lambda^k_{p+1}-\lambda^k_p)}{\sum_{i=1}^{N_f}\Tilde{\gamma}_i \lambda^i_{p+1}} \uu_h^-({\bf x}_k^\kappa,t^{(n)}).
\end{equation*}

\noindent with $N_f+1$ positive weights such that 
\begin{equation*}
\frac{\sum_{i=1}^{N_f}\Tilde{\gamma}_i \lambda^i_{p}}{\sum_{i=1}^{N_f}\Tilde{\gamma}_i \lambda^i_{p+1}}+ \sum_{k=1}^{N_f} \frac{\Tilde{\gamma}_k^\kappa (\lambda^k_{p+1}-\lambda^k_p)}{\sum_{i=1}^{N_f}\Tilde{\gamma}_i \lambda^i_{p+1}}=1,
\end{equation*}

\noindent and by recurrence $\uu^\star_{p+1}$ is also a convex combination of $\uu^\star_1$ and the $\uu_h^-({\bf x}_k^\kappa,t^{(n)})$. Therefore, if $\uu^\star_1$ is in $B$, then $\uu^\star_p$ is also in $\B$ for all $p \geq 1$ and stays in a compact subset of $\B$ which ensures that the $\lambda^k_p$ are bounded. So they converge to some finite $\lambda^{\star k}_\kappa$ and $\uu^\star_p$ converges to
\begin{equation}
\label{eqn:defustarloc}
\uu^\star_\kappa=\sum_{k=1}^{N_f} \frac{\Tilde{\gamma}_k^\kappa \lambda^{\star k}_\kappa}{\sum_{i=1}^{N_f}\Tilde{\gamma}_i \lambda^{*i}_\kappa} \Big(\uu_h^-({\bf x}_k^\kappa,t^{(n)})-\frac{{\bf f}\big(\uu_h^-({\bf x}_k^\kappa,t^{(n)})\big)\cdot \n_k}{\lambda^{\star k}_\kappa}\Big) \in \B.
\end{equation}

By definition of the $\lambda_p^k$, \cref{eqn:loclambar} holds and with the definition of $\uu^\star_\kappa$, \cref{eqn:locubar} is also satisfied, while \cref{thm:CFLHO} still holds with $\lambda_\kappa^{\star k}$ instead of $\lambda_\kappa^\star$ in \cref{eqn:CFL}. We summarize these results in the following lemma.

\begin{lemma}
\label{lem:barloc}
Suppose that the numerical scheme satisfies \cref{eqn:meaneqngen}, \cref{eqn:surfquad} and \cref{eqn:closesurface}, let $\B$ be a invariant domain and suppose that for all $1\leq i\leq N_f$, $\uu_h^\pm(x_i^\kappa,t^{(n)})$ are in $\B$, then  $\uu^\star_\kappa$ defined by \cref{eqn:defustarloc} is in $\B$, satisfies \cref{eqn:locubar}, and there exists a family of finite and positive estimates $({\lambda}_\kappa^{\star k})_{1\leq k\leq N_f}$ satisfying \cref{eqn:loclambar}.
\end{lemma}

%
%
\section{Examples of high-order spectral discontinuous methods}\label{sec:schemes}

We here review some high-order spectral discontinuous approximations of \cref{eqn:HCL} which satisfy the assumptions of discrete conservation \cref{eqn:meaneqngen}, existence of quadrature rule \cref{eqn:surfquad} and preservation of uniform states \cref{eqn:closesurface}. As a consequence there exists pseudo-equilibrium states $\uu_\kappa^\star$ such that \cref{lem:bar}, \cref{lem:barloc}, and \cref{thm:CFLHO} hold, and the limiter \cref{eqn:IRPlim} can be applied. In the following,  we  consider  a  partition $\T_h$ of $D\subset\mathbb{R}^d$, composed  of non-overlapping and non-empty elements $\kappa$, and by ${\cal F}_h$ we denote the set of faces in the partition. The approximate solution is sought under the form
\begin{equation}\label{eq:DG_num_sol}
 {\bf u}_h({\bf x},t) = \sum_{k=1}^{N_p} \phi^\kappa_{k}({\bf x}){\bf U}^\kappa_{k}(t) \quad \forall{\bf x}\in\kappa,\, \kappa\in \T_h,\, \forall t\geq0,
\end{equation}

\noindent where the basis functions $\phi^\kappa_{k}$ span the function space ${\cal V}_h^p(\kappa)$ restricted onto $\kappa$ and the discrete scheme may be written as 
\begin{equation}\label{eqn:fully-discr_DG}
 M_k^\kappa \frac{{\bf U}_{k,n+1}^\kappa-{\bf U}_{k,n}^\kappa}{\Delta t^{(n)}} + {\bf R}_{k}^\kappa({\bf u}_h^{(n)}) = 0 \quad \forall \kappa\in \T_h, \; 1\leq k\leq N_p, \; n\geq 0,
\end{equation}

\noindent where ${\bf U}_{k,n}^\kappa={\bf U}_{k}^\kappa(t^{(n)})$ and the $M_k^\kappa$ are the entries of the mass matrix.

\subsection{Discontinuous Galerkin Method}\label{sec:modalDG}
The first numerical scheme we describe here is the discontinuous Galerkin method with modal basis \cite{DG,bassi1997high,cockburn-shu89}. We look for approximate solutions in the function space of discontinuous polynomials
\begin{equation*}
 {\cal V}_h^p=\bigoplus\limits_{\kappa\in\T_h}{\cal V}_h^p(\kappa)=\{\phi\in L^2(D):\;\phi|_{\kappa}\circ{\bf x}_\kappa\in\Po^p(\hat{K})\; \forall\kappa\in \T_h\},
\end{equation*}

\noindent where $\Po^{p}(\hat{K})$ is a polynomial space over a reference element $\hat{K}$, Each physical element $\kappa$ is the image of $\hat{K}$ through the mapping ${\bf x}={\bf x}_\kappa(\bxi)$ with $\bxi=(\xi_1,\dots,\xi_d)$. Likewise, each face $f$ in ${\cal F}_h$ is the image of a reference face $\hat{F}$ through the mapping ${\bf x}={\bf x}_f(\xi_1,\dots,\xi_{d-1})$. We suppose that we have a quadrature $(\bxi_i^V, \omega_i^V)_{1\leq i\leq N_v}$ on $\hat{K}$ and denote ${\bf y}_i^{\kappa}={\bf x}_\kappa(\bxi_i)$  (see \cref{fig:stencil_2D}). Similarly, we consider a quadrature $(\bxi_k^f,\omega_k^f)_{1\leq k \leq n_f}$ on $\hat{F}$ and denote the faces of $\kappa$ $(f_\kappa^j)_{1\leq j \leq N_\kappa}$ where $N_\kappa$ is the number of faces of $\kappa$ and the $f_\kappa^j \in {\cal F}$ are distinct. Then define the ${\bf x}_i^\kappa$ in \cref{eqn:meaneqngen} by ${\bf x}_{(j-1)n_f+k}^\kappa={\bf x}_{f^j_\kappa}(\bxi_k^{f^j_\kappa})$ for $1 \leq j \leq N_\kappa$ and $1\leq k \leq n_f$ and $N_f=N_\kappa \times n_f$. We further define Jacobians of the transformations by $J_\kappa({\bf x})=|{\bf x}_\kappa'(\bxi)|$ and $J_f({\bf x})=|{\bf x}_f'(\xi_1,\dots,\xi_{d-1})|$.

The DG method consists in defining a discrete weak formulation of $\cref{eqn:HCL}$ by multiplying it with test functions $\phi_k^\kappa$ spanning ${\cal V}_h^p$ and integrating over $\kappa$, using integration by parts and approximating ${\bf f}(\uu_h)\cdot\n$ by two-point numerical fluxes and the integrals by the quadrature rules. The space discretization in \cref{eqn:fully-discr_DG} reads
\begin{multline}\label{eqn:weakDG}
 {\bf R}_{k}^\kappa({\bf u}_h) = 
    - \sum_{i=1}^{N_v} \omega_i^V J_\kappa({\bf y}_i^\kappa) {\bf f}\big(\uu_h({\bf y}_i^\kappa,t^{(n)})\big)\cdot \nabla \phi_k^\kappa({\bf y}_i^\kappa) \\ + \sum_{i=1}^{N_f} \omega_i^f J_f({\bf x}_i^\kappa) {\bf h}\big(\uu_h^-({\bf x}_i^\kappa,t^{(n)}),\uu_h^+({\bf x}_i^\kappa,t^{(n)}),\n_i^\kappa\big) \phi_k^\kappa({\bf x}_i^\kappa).
\end{multline}

We use an orthonormal basis such that $M_\kappa^k=|\kappa|:=\sum_{i=1}^{N_v} \omega_i^V J_\kappa({\bf y}_i^\kappa)$ and further seting $\phi_k^\kappa=1_\kappa$ the indicator function of $\kappa$, the first sum vanishes and we obtain \cref{eqn:meaneqngen} with $s_k^\kappa=\frac{\omega_k^f J_f({\bf x}_k^\kappa)}{|\kappa|}$ and $\beta_k^\kappa=\varepsilon_\kappa s_k^\kappa$ in \cref{eqn:surfquad} where $\varepsilon_\kappa$ is evaluated from \cref{eq:def_s_over_beta}. Then by \cref{eqn:CFLmodal}, the scheme \cref{eqn:fully-discr_DG,eqn:weakDG} is IDP under the condition 
\begin{equation}\label{eq:CFL_modalDG}
     \Delta t^{(n)} \max_{\kappa \in \T_h} \frac{1}{\varepsilon_\kappa} \max_{1\leq k\leq N_f} \max\Big(\lambda^\star_\kappa, |\lambda|\big(\uu_h^-({\bf x}_k^\kappa,t^{(n)}),\uu_h^+({\bf x}_k^\kappa,t^{(n)})\big)\Big)\leq \frac{1}{2}.
\end{equation}




%
\subsection{Discontinuous Galerkin Spectral Element Method}\label{sec:DGSEM}

In the DGSEM, the reference element is an hypercube : $\hat{K} = I^d := \{ \bxi =( \xi_1, \dots, \xi_d) : \; -1 \leq \xi_j \leq 1 \}$ and the polynomial space $\Po^{p}(I^d)$ is formed by tensor products of polynomials of degree at most $p$ in each direction. The approximate solution is sought under the form \cref{eq:DG_num_sol} where $({\bf U}^\kappa_k)_{1\leq k\leq N_p}$ are the $N_p=(p+1)^d$ DOFs in the element $\kappa$ with indexing
\begin{equation*}
 k=\ol{k}(i_1,\dots,i_d):=1+\sum_{j=1}^di_j(p+1)^{j-1} \quad 0\leq i_1,\dots,i_d \leq p.
\end{equation*}

We define a basis $(\phi^\kappa_k)_{1\leq k\leq N_p}$ of ${\cal V}_h^p(\kappa)$ by using tensor products: $\phi^\kappa_k({\bf x})=\phi^\kappa_k({\bf x}_\kappa(\bxi))=\Pi_{j=1}^d\ell_{i_j}(\xi_j)$, where $\ell_{0\leq i\leq p}$ denote the $i$th Lagrange interpolation polynomial associated to $\zeta_i$ the $i$th Gauss-Lobatto quadrature node with $\zeta_0=-1<\zeta_1<\dots<\zeta_p=1$ (see \cref{fig:stencil_2D_DGSEM}) and by $\omega_i$ we denote the associated weight. We therefore have the following cardinality relation at quadrature points $\bxi_{k'}=(\xi_{i'_1},\dots,\xi_{i'_d})$ in $\hat{K}$: $\phi^\kappa_k({\bf x}_{k'}^\kappa)=\phi^\kappa_k({\bf x}_\kappa(\bxi_{k'}))=\delta_{i_1,i'_1}\dots\delta_{i_d,i'_d}$ with $\delta_{i,i'}$ the Kronecker symbol, so the DOFs correspond to the point values of the solution: ${\bf U}_k^\kappa(t)={\bf u}_h({\bf y}_k^\kappa,t)$ and interpolation and quadrature points are collocated, hence $N_v=N_p$.


%
%

%
\begin{figure}[ht]
\begin{center}
\begin{tikzpicture}
[declare function={la(\x) = (\x+1/sqrt(5))/(-1+1/sqrt(5)) * (\x-1/sqrt(5))/(-1-1/sqrt(5)) * (\x-1)/(-2);
                   lb(\x) = (\x+1)/(-1/sqrt(5)+1) * (\x-1/sqrt(5))/(-2/sqrt(5)) * (\x-1)/(-1/sqrt(5)-1);
                   lc(\x) = (\x+1)/(1/sqrt(5)+1) * (\x+1/sqrt(5))/(2/sqrt(5)) * (\x-1)/(1/sqrt(5)-1);
                   ld(\x) = (\x+1)/(2) * (\x+1/sqrt(5))/(1+1/sqrt(5)) * (\x-1/sqrt(5))/(1-1/sqrt(5));
                   }]
\draw (1.2,1.625) node {$\kappa=\kappa^-$};
\draw (5.,1.65)   node {$\kappa^+$};
\draw (2.8,3.)    node[above] {$f$};
\draw (2.86,2.20) node[below left]  {${\bf u}_h^-$};
\draw (2.86,2.20) node[above right] {${\bf u}_h^+$};
\draw [>=stealth,->] (2.94,0.90) -- (3.9,0.95) ;
\draw (3.9,0.95) node[below right] {$\n_k^\kappa$};
\draw (2.94,0.90) node[above right] {${\bf x}_k^\kappa$};
\draw (0.21,2.59) node[above left] {${\bf y}_i^\kappa$};
\def\xad{-1.00};  \def\yad{3.40}; \def\xbd{0.05}; \def\ybd{3.49}; \def\xcd{1.75}; \def\ycd{3.11}; \def\xdd{2.80}; \def\ydd{3.00};
\def\xac{-0.723}; \def\yac{2.46}; \def\xbc{0.21}; \def\ybc{2.59}; \def\xcc{1.87}; \def\ycc{2.27}; \def\xdc{2.86}; \def\ydc{2.20};
\def\xab{-0.48};  \def\yab{0.94}; \def\xbb{0.45}; \def\ybb{1.03}; \def\xcb{2.05}; \def\ycb{0.91}; \def\xdb{2.94}; \def\ydb{0.90};
\def\xaa{0.00};   \def\yaa{0.00}; \def\xba{0.83}; \def\yba{-0.03}; \def\xca{2.17}; \def\yca{0.01}; \def\xda{3.00}; \def\yda{0.10};
\draw (\xad,\yad) node {$\bullet$}; \draw (\xbd,\ybd) node {$\bullet$}; \draw (\xcd,\ycd) node {$\bullet$}; \draw (\xdd,\ydd) node {$\bullet$};
\draw (\xac,\yac) node {$\bullet$}; \draw (\xbc,\ybc) node {$\bullet$}; \draw (\xcc,\ycc) node {$\bullet$}; \draw (\xdc,\ydc) node {$\bullet$};
\draw (\xab,\yab) node {$\bullet$}; \draw (\xbb,\ybb) node {$\bullet$}; \draw (\xcb,\ycb) node {$\bullet$}; \draw (\xdb,\ydb) node {$\bullet$};
\draw (\xaa,\yaa) node {$\bullet$}; \draw (\xba,\yba) node {$\bullet$}; \draw (\xca,\yca) node {$\bullet$}; \draw (\xda,\yda) node {$\bullet$};
\draw [domain=-1:1] plot ({la(\x)*\xaa+lb(\x)*\xba+lc(\x)*\xca+ld(\x)*\xda}, {la(\x)*\yaa+lb(\x)*\yba+lc(\x)*\yca+ld(\x)*\yda});
\draw [domain=-1:1] plot ({la(\x)*\xad+lb(\x)*\xbd+lc(\x)*\xcd+ld(\x)*\xdd}, {la(\x)*\yad+lb(\x)*\ybd+lc(\x)*\ycd+ld(\x)*\ydd});
\draw [domain=-1:1] plot ({la(\x)*\xda+lb(\x)*\xdb+lc(\x)*\xdc+ld(\x)*\xdd}, {la(\x)*\yda+lb(\x)*\ydb+lc(\x)*\ydc+ld(\x)*\ydd});
\draw [domain=-1:1] plot ({la(\x)*\xaa+lb(\x)*\xab+lc(\x)*\xac+ld(\x)*\xad}, {la(\x)*\yaa+lb(\x)*\yab+lc(\x)*\yac+ld(\x)*\yad});
%
\draw [>=stealth,-] (3.,0.1) -- (6.5,-0.5) ;
\draw [>=stealth,-] (6.5,-0.5) -- (7.5,4.) ;
\draw [>=stealth,-] (7.5,4.) -- (2.8,3.) ;
\end{tikzpicture}
\caption{Notations for the DGSEM and $d=2$: inner and outer elements, $\kappa^-$ and $\kappa^+$, for $d=2$; definitions of traces ${\bf u}_h^\pm$ on the interface $f$ and of the unit outward normal vector $\n_k^\kappa$. Element quadrature nodes ${\bf y}_i^\kappa$ and surface quadrature node ${\bf x}_k^\kappa$ that are also included in the ${\bf y}_i^\kappa$.}
\label{fig:stencil_2D_DGSEM}
\end{center}
\end{figure}
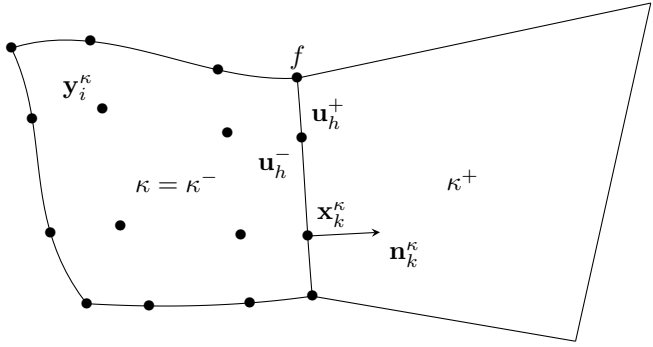

Let introduce the discrete derivative matrix with entries $D_{ij} = \ell_j'(\zeta_i)$ with $0\leq i,j \leq p$. The DGSEM discretization takes the form \cref{eqn:fully-discr_DG} with $M_k^\kappa=\omega_k^V J_\kappa({\bf k}_k^\kappa)$, $\omega_k^V=\Pi_{j=1}^d\omega_{i_j}$ for $k=\ol{k}(i_1,\dots,i_d)$ and 
\begin{align}
  {\bf R}_k^\kappa({\bf u}_h) &= 2 \omega_k^V \sum_{j=1}^d \sum_{l=0}^p D_{i_j l}{\bf h}_{sym}\big({\bf U}_k^\kappa,{\bf U}_{k'_j}^\kappa,\{J_\kappa\nabla\xi_j \}_{(k,k'_j)}\big) \nonumber\\
	&+ \sum_{i=1}^{N_f}\phi_k^\kappa({\bf x}_i^\kappa)\omega_i^fJ_f({\bf x}_i^\kappa) \Big( {\bf h}\big({\bf U}_{k,n}^\kappa,{\bf u}_h^+({\bf x}_i^\kappa,t),{\bf n}_i^\kappa\big)-{\bf f}\big({\bf U}_{k,n}^\kappa\big)\cdot{\bf n}_i^\kappa \Big)
 \label{eq:semi-discr_DGSEM-res}
\end{align}

\noindent where for $d=3$ $k'_1=\ol{k}(l,i_2,i_3)$, $k'_2=\ol{k}(i_1,l,i_3)$ and $k'_3=\ol{k}(i_1,i_2,l)$, hence $k'_j=k+(i'_j-i_j)(p+1)^{j-1}$, $N_f=2d(p+1)^{d-1}$, $\omega_i^f=\Pi_{j=1}^{d-1}\omega_{i_j}$, and 
\begin{equation*}\label{eq:def_Jkappa_mean}
\begin{array}{c}
 \{J_\kappa \nabla {\bf \xi} \}_{(k,k'_j)}=\frac{1}{2} \big(J_\kappa ( {\bf y}_k^\kappa) \nabla \xi_j (\bxi_k) + J_\kappa({\bf y}_{k'_j}^\kappa) \nabla \xi_j(\bxi_{k'_j})\big), 
 \end{array}
\end{equation*}

\noindent have been introduced to keep conservation of the scheme \cite{wintermeyer_etal_17}. By ${\bf h}_{sym}$ we denote a two-point flux supposed to be symmetric in the sense that ${\bf h}_{sym}({\bf u},{\bf v}, \n)={\bf h}_{sym}({\bf v},{\bf u}, \n)$. Note that $\phi_k^\kappa({\bf x}_i^\kappa)=1$ if ${\bf x}_i^\kappa={\bf y}_k^\kappa$ and $\phi_k^\kappa({\bf x}_i^\kappa)=0$ else.

The choice of the ${\bf y}_k^\kappa$ in the quadrature \cref{eqn:surfquad} is not unique and we here use the following decomposition
\begin{align*}
 \langle{\bf u}_h^{(n)}\rangle_\kappa &= \frac{1}{|\kappa|}\sum_{{\bf y}_i^\kappa\in\kappa} \omega_{i}^V J_\kappa ({\bf y}_i^\kappa){\bf U}_{i,n}^\kappa \\
 &= \frac{1}{|\kappa|}\sum_{{\bf y}_i^\kappa\in\text{int}(\kappa)} \omega_{i}^V J_\kappa ({\bf y}_i^\kappa){\bf U}_{i,n}^\kappa  + \frac{1}{|\kappa|}\sum_{i=1}^{N_f} \Tilde\omega_i^f J_\kappa({\bf x}_i^\kappa){\bf u}_h({\bf x}_i^\kappa,t^{(n)}),
\end{align*}

\noindent where $\text{int}(\kappa)$ denotes the interior of $\kappa$, while $\Tilde\omega_i^f J_\kappa({\bf x}_i^\kappa)=\frac{1}{d}\omega_{j}^V J_\kappa ({\bf y}_j^\kappa)$ if ${\bf y}_j^\kappa={\bf x}_i^\kappa$ is a vertex of the $d$-dimensional hexahedron, $\Tilde\omega_i^f J_\kappa({\bf x}_i^\kappa)=\frac{1}{d-1}\omega_{j}^V J_\kappa ({\bf y}_j^\kappa)$  if ${\bf y}_j^\kappa={\bf x}_i^\kappa$ is on some edge, and  $\Tilde\omega_i^f J_\kappa({\bf x}_i^\kappa)=\omega_{j}^V J_\kappa ({\bf y}_j^\kappa)$ else.

Summing \cref{eqn:fully-discr_DG} over $1\leq k\leq N_v$ gives for the cell-averaged solution
\begin{align*}
 \langle{\bf u}_h^{(n+1)}\rangle_\kappa &= \langle{\bf u}_h^{(n)}\rangle_\kappa -\frac{\Delta t^{(n)}}{|\kappa|}\sum_{k=1}^{N_v}{\bf R}^\kappa_{k}({\bf u}_h^{(n)}) \\
 &= \langle{\bf u}_h^{(n)}\rangle_\kappa - \frac{\Delta t^{(n)}}{|\kappa|}\sum_{i=1}^{N_f}\omega_i^fJ_f({\bf x}_i^\kappa){\bf h}\big({\bf u}_h^-({\bf x}_i^\kappa,t^{(n)}),{\bf u}_h^+({\bf x}_i^\kappa,t^{(n)}),\n_i^\kappa\big), 
\end{align*}

\noindent by conservation of the DGSEM \cite{winters_etal_16,wintermeyer_etal_17} and providing that the so-called metric identities are satisfied at the discrete level \cite{kopriva_metric_id_06}. This relation can be identified with \cref{eqn:meaneqngen} with $s_k^\kappa=\frac{\omega_k^fJ_f({\bf x}_k^\kappa)}{|\kappa|}$. Then we can apply \cref{thm:CFLHO} with $\beta^\kappa_k=\frac{\Tilde\omega_k^f J_\kappa({\bf x}_k^\kappa)}{|\kappa|}$ and the DGSEM scheme \cref{eqn:fully-discr_DG} is IDP under the condition 
\begin{equation}\label{eq:CFL_DGSEM}
     \Delta t^{(n)} \max\limits_{\kappa\in\Omega_h} \max\limits_{1\leq k\leq N_f} \frac{\omega_k^f J_f({\bf x}_k^\kappa)}{\Tilde{\omega}_k^f J_\kappa({\bf x}_k^\kappa)} \max\Big(\lambda^\star_\kappa, |\lambda|\big(\uu_h^-({\bf x}_k^\kappa,t^{(n)}),\uu_h^+({\bf x}_k^\kappa,t^{(n)})\big)\Big)\leq \frac{1}{2}.
\end{equation}

\subsection{Other methods}

Properties \cref{eqn:meaneqngen,eqn:surfquad,eqn:closesurface} also hold for other discretely conservative spectral difference methods on general curved elements provided the discretization operators satisfy the metric identities at the discrete level, which imposes some limits on the order of approximation of the mesh elements compared to the approximation order of the solution \cite{kopriva_metric_id_06}. The limiter \cref{eqn:IRPlim} may hence be applied to make these methods invariant domain preserving. We list some examples below.

The skew-symmetric entropy stable modal discontinuous Galerkin methods \cite{chan_skewDG_2019} uses skew-hybrized summation-by-parts (SBP) operators allowing conservation and free-stream preservation under the standard accuracy requirements of volume and surface quadratures. In \cite{Crean_etal_SBP_curved_2018} multidimensional discretization schemes based on SBP operators on general curved elements generalize the staggered finite differences from \cite{fisher_carpenter_13}. The discretizations with curved elements remain accurate, conservative, and entropy stable. Spectral differences \cite{Liu_etal_SDM_06} and staggered Chebyshev \cite{KoprivaKolias_stagg_96} methods on curved elements belong to a same family of conservative approximations satisfying the discrete metric identities. The methods use two sets of interpolation points for the solution and fluxes and impose the discrete residuals to be satisfied at solution points. Flux points contain points at faces of the elements where two-point numerical fluxes are used. Then, the space derivatives at solution points are evaluated by differencing the polynomials interpolating the fluxes.

%
%
\section{Numerical experiments}
\label{sec:experiments}

Let consider the compressible Euler equations of gas dynamics. The conservative variables and fluxes in \cref{eqn:HCL} are 
\begin{equation}
\label{eqn:Euler}
\uu =
\begin{pmatrix}
    \rho \\
    \rho {\bf v} \\
    \rho E \\
\end{pmatrix},
\quad
{\bf f }=
\begin{pmatrix}
    \rho {\bf v}^\top \\
    \rho {\bf v}{\bf v}^\top + \p {\bf I}_d \\
    (\rho E + \p) {\bf v}^\top\\
\end{pmatrix},
\end{equation} 

\noindent where $\rho$, ${\bf v}$, and $E$ denote the density, velocity vector, and specific total energy, respectively. The system is closed by defining the equation of state $\p=\p(\frac{1}{\rho},e)$ with $e=E-\frac{1}{2}{\bf v}\cdot{\bf v}$ the specific internal energy and the system is hyperbolic over the set of states $\Omega^a=\{\uu\in\mathbb{R}^{d+2}:\; \rho>0, {\bf v}\in\mathbb{R}^d, e>0\}$. We focus here on the polytropic ideal gas law $\p = (\gamma-1)\rho e$ where $\gamma=\frac{C_p}{C_v}=\frac{7}{5}$ is the ratio of specific heats. The compressible Euler equations \cref{eqn:HCL,eqn:Euler} possess the natural entropy -- entropy flux pair
\begin{equation*}
    \eta=-\rho \s, \quad {\bf q}= -\rho\s{\bf v}, \quad \s=C_v\ln\Big(\frac{p}{\rho^\gamma}\Big),
\end{equation*}

\noindent and $\B=\{\uu\in\Omega^a\;\s(\uu)\geq\s_0\}$, with $\s_0$ in $\mathbb{R}$, is an invariant domain for \cref{eqn:HCL,eqn:Euler} \cite{Frid_idp_LF_01}. We use our convex limiting strategy with the quasiconcave functions $\psi_1\equiv\rho$ and $\psi_2\equiv\rho e$. 

We now test the robustness and efficiency of the CFL condition on simulations of \cref{eqn:HCL,eqn:Euler} with discontinuous solutions on one-dimensional and unstructured two-dimensional grids. We use the modal DG method in \cref{sec:modalDG} and the DGSEM in \cref{sec:DGSEM}. The CFL conditions \cref{eq:CFL_modalDG,eq:CFL_DGSEM} guaranty the cell-averaged solution to be IDP and we then apply the limiter \cref{eqn:IRPlim} to further impose the high-order solution to be IDP. We will compare the following limiting strategies: a positivity limiter (POS) which imposes the solution to remain in $\Omega^a$ thus extending \cite{zhang2010positivity} to unstructured grids; an IDP limiter which imposes the solution to remain in the convex hull of the states in \cref{eq:convex_comb_un+1} and computing the time step in \cref{eq:CFL_modalDG,eq:CFL_DGSEM} with either the global wave estimate $\lambda_\kappa^\star$ (IDP) from algorithm \cref{eqn:modifseq}, or the local wave estimate $\lambda_\kappa^\star$ from \cref{eqn:modifseq_loc} (IDPloc). Imposing the IDP property may result in over-limiting of the solution and some strategies are usually applied such as bound relaxation \cite{Guermond_etal_IDP_conv_lim_19}, or subcell smoothness indicator \cite{PAZNER_idg_DGSEM20211}. We here follow the second strategy which relies on the smoothness indicator from \cite{persson_peraire_lim_06} (see \cite[Sec.~4.4]{PAZNER_idg_DGSEM20211} for details). Finally, we use the Suliciu pressure relaxation based numerical flux from \cite[Sec.~2.4.6]{bouchut_04} at interfaces, while for ${\bf h}_{sym}$ in the DGSEM scheme \cref{eq:semi-discr_DGSEM-res} we use the Kennedy and Grubber splitting from \cite{winters_etal_16}.

\subsection{Riemann problems}

We here consider Riemann problems \cref{eqn:RP} with initial data given in \cref{tab:RP_IC_LP}. We first consider computations with the DGSEM (see \cref{sec:schemes}) and the three limiting strategies. Results are shown in \cref{fig:sod,fig:lax,fig:toro}. The POS limiter only ensures that the solution remains in the set of states $\Omega^a$ and does not modify non-physical oscillations, with the IDP and IDPloc strategies succeed in damping spurious oscillations and result in very close oscillations. All numerical experiments always show that there is no sensible improvement to evaluate the pseudo-equilibrium state $\uu_\kappa^\star$ with local wave estimates $\lambda_\kappa^{\star k}$ in \cref{eqn:modifseq} instead of a global estimate $\lambda_\kappa^{\star}$ in \cref{eqn:modifseq_loc}, while it leads to a more expensive algorithm. Besides, our observation show that using a local estimates require more iterations for algorithm \cref{eqn:modifseq_loc} to converge with a global average between $2.8$ and $3.2$ iterations evazluated over the whole computations compared to between $1.11$ and $1.17$ when using \cref{eqn:modifseq}. In the latter case, for most of the computations the initial guess $\uu^\star_0$ given in \cref{eqn:modifseq} satisfies the requirement \cref{eqn:locubar} and \cref{eqn:loclambar} so no more step is needed. 

\begin{table}
     \begin{center}
     \caption{Initial conditions of Riemann problems \cref{eqn:RP} where $x_0$ indicates the abscissa separating the states.}
     \begin{tabular}{lcccc}
        \noalign{\smallskip}\hline\noalign{\smallskip}
     problem & left state $(\rho_L,u_L,p_L)^\top$ & right state
$(\rho_R,u_R,p_R)^\top$ & $x_0$ \\ 
        \noalign{\smallskip}\hline\noalign{\smallskip}
    Sod \cite{sod1978survey} & $(1,0,1)^\top$ & $(0.125,0,0.1)^\top$ & $0$ \\ 
    Lax & $(0.445,0.698,3.528)^\top$ & $(0.5,0,0.571)^\top$ & $0$ \\ 
    Toro $4$ \cite{toro_book} & $(5.99924,19.5975,460.894)^\top$ &
$(5.99242,-6.19633,46.0950)^\top$ & $-0.1$ \\ 
        \noalign{\smallskip}\hline\noalign{\smallskip}
    \end{tabular}
     \label{tab:RP_IC_LP}
    \end{center}
\end{table}

\begin{figure}
    \centering
    \includegraphics[width=4.6cm]{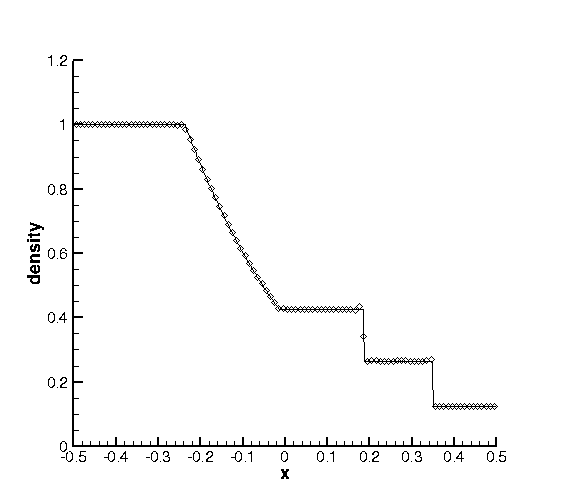}\hspace{-0.6cm}
    \includegraphics[width=4.6cm]{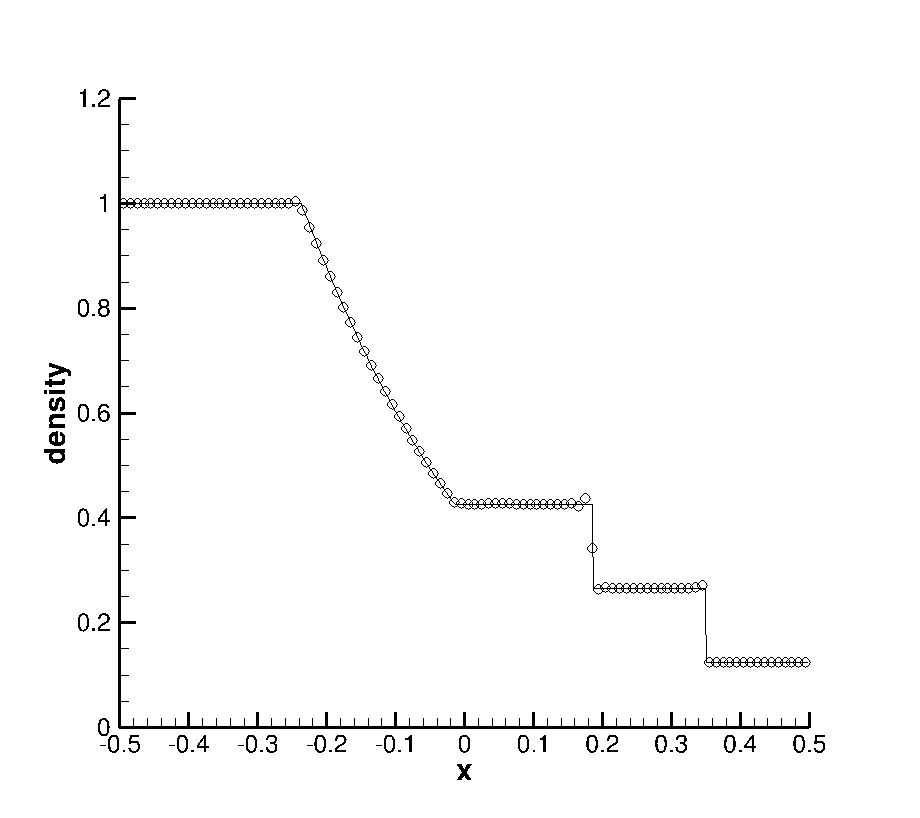}\hspace{-0.6cm}
    \includegraphics[width=4.6cm]{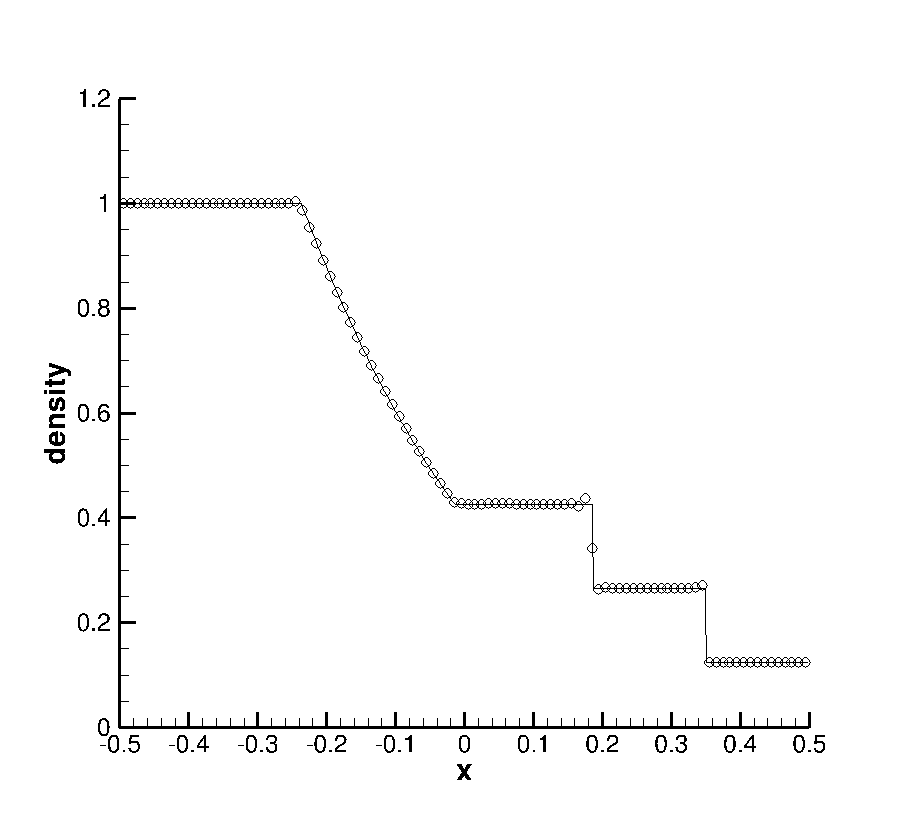} \\
    \includegraphics[width=4.6cm]{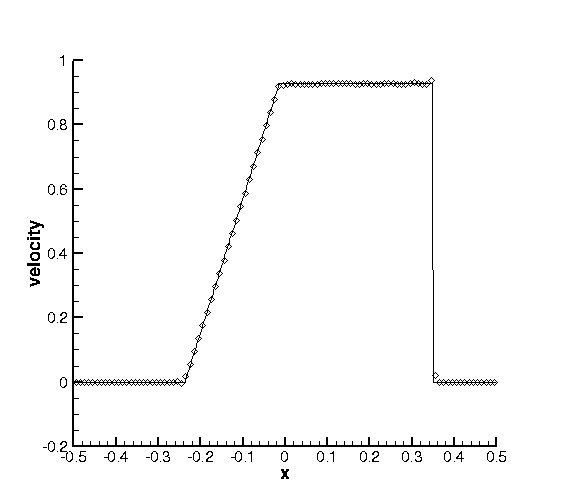}\hspace{-0.6cm}
    \includegraphics[width=4.6cm]{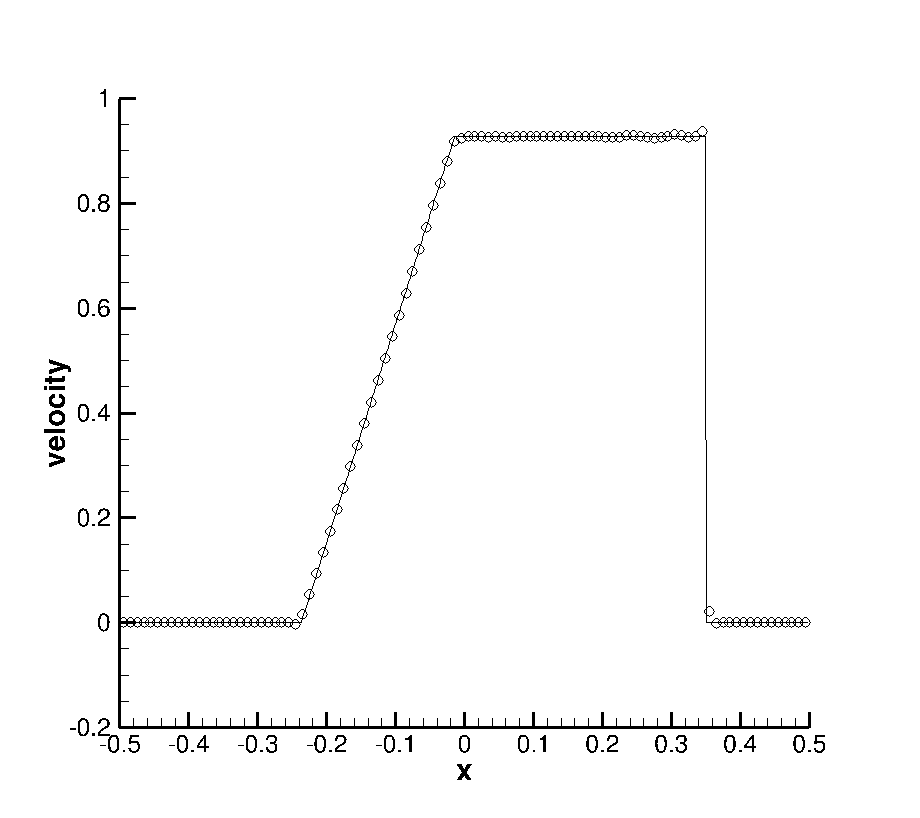}\hspace{-0.6cm}
    \includegraphics[width=4.6cm]{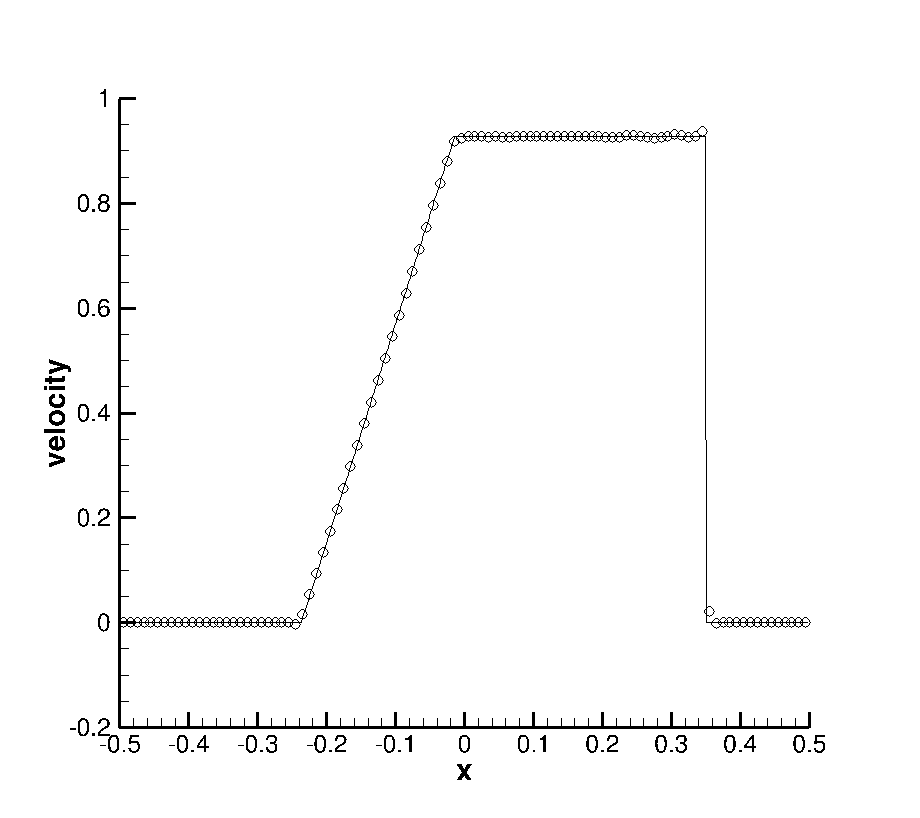} \\
    \subfloat[POS]{\includegraphics[width=4.6cm]{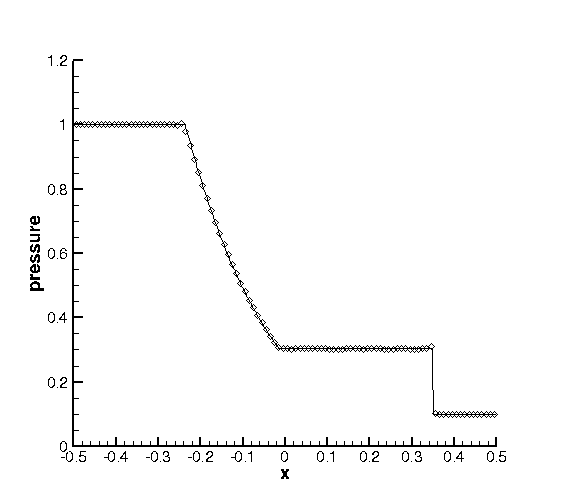}}\hspace{-0.6cm}
    \subfloat[IDP]{\includegraphics[width=4.6cm]{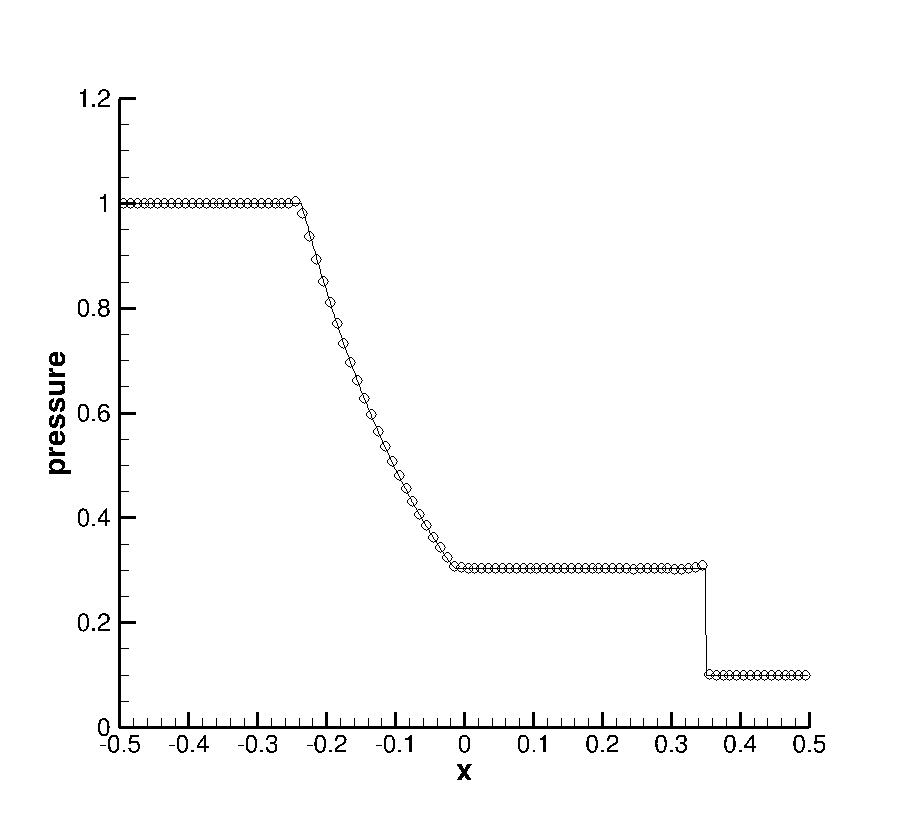}}\hspace{-0.6cm}
    \subfloat[IDPloc]{\includegraphics[width=4.6cm]{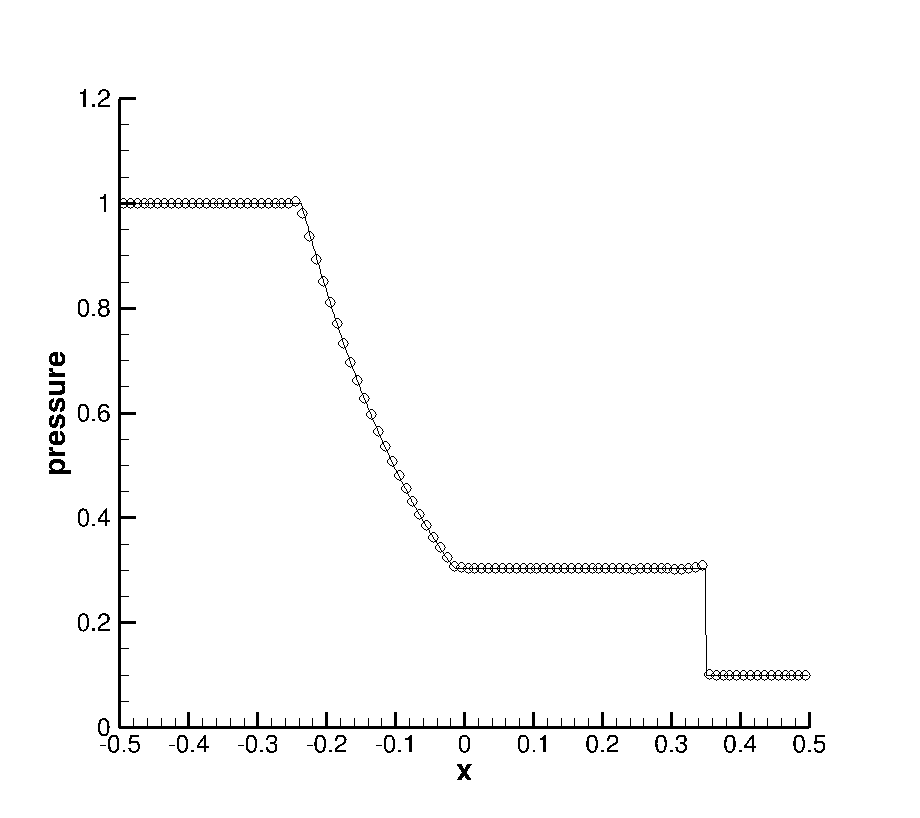}}
    \caption{DGSEM computations of the Sod problem at $t=0.2$ with $p=3$ and $N=100$ elements for the density (top), velocity (middle), and pressure (bottom).}
    \label{fig:sod}
\end{figure}


\begin{figure}
    \centering
    \includegraphics[width=4.6cm]{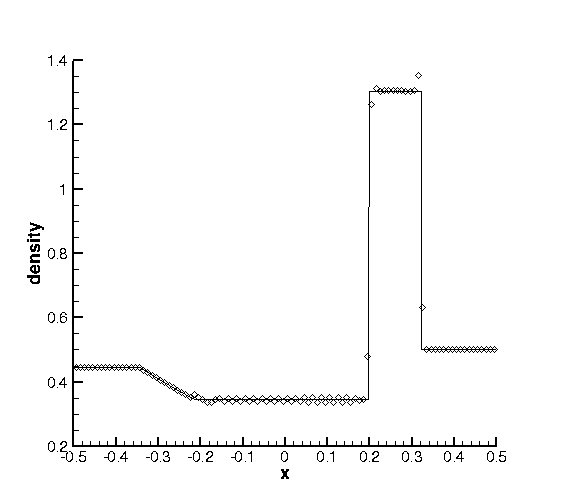}\hspace{-0.6cm}
    \includegraphics[width=4.6cm]{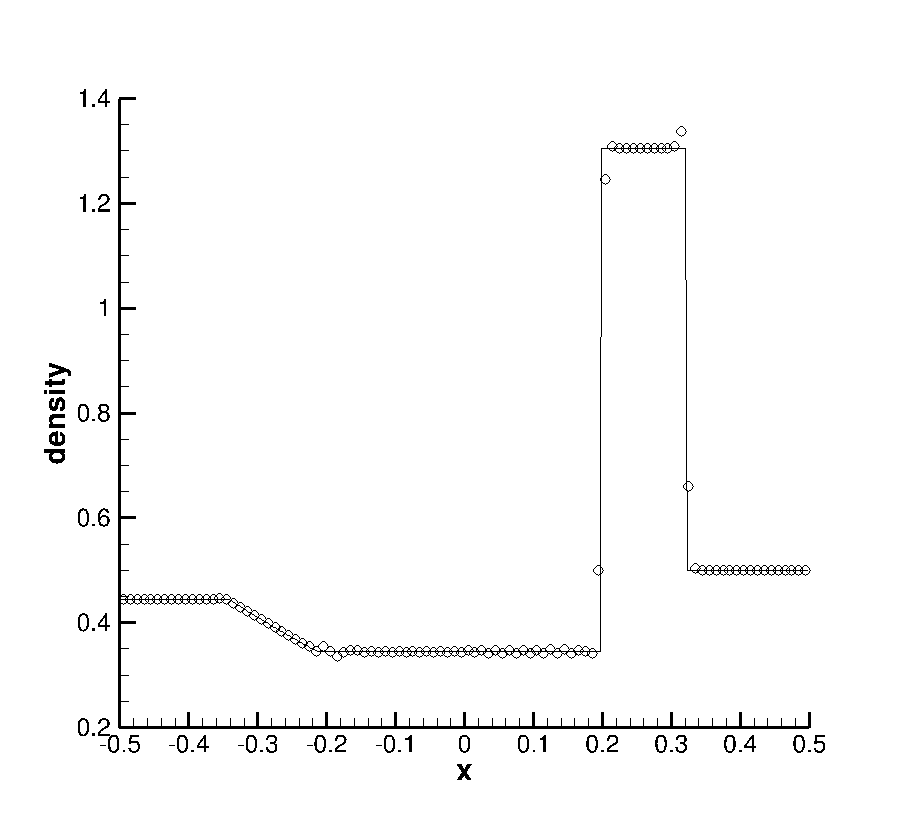}\hspace{-0.6cm}
    \includegraphics[width=4.6cm]{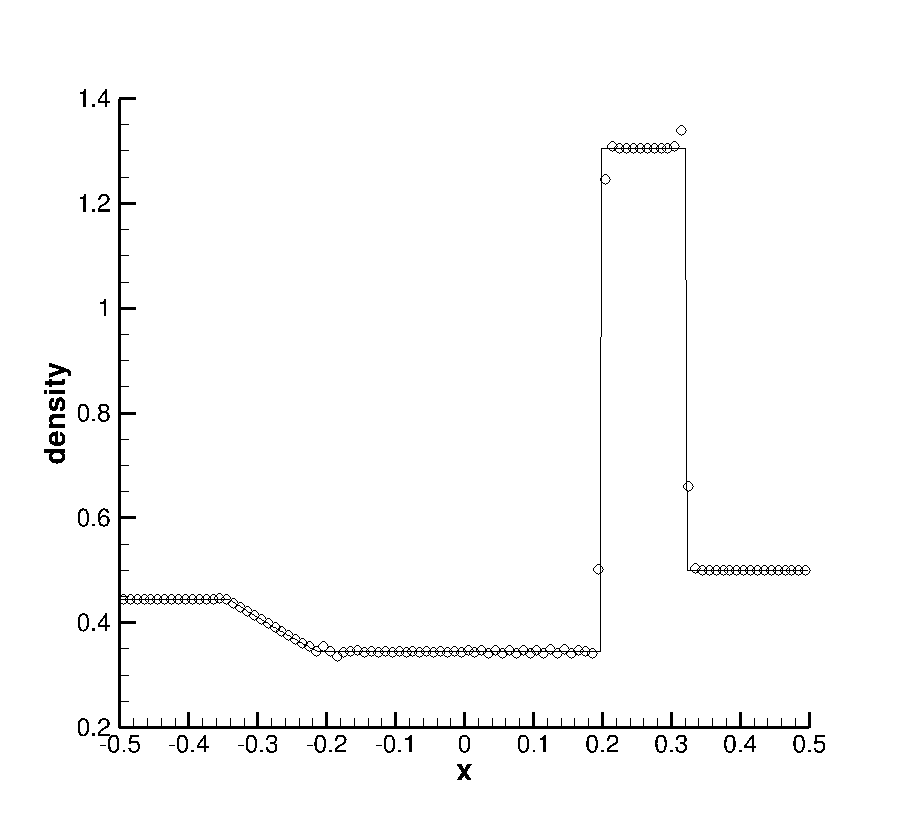} \\
    \includegraphics[width=4.6cm]{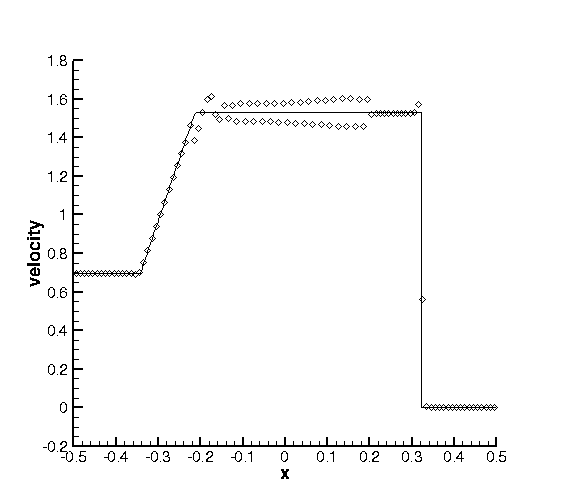}\hspace{-0.6cm}
    \includegraphics[width=4.6cm]{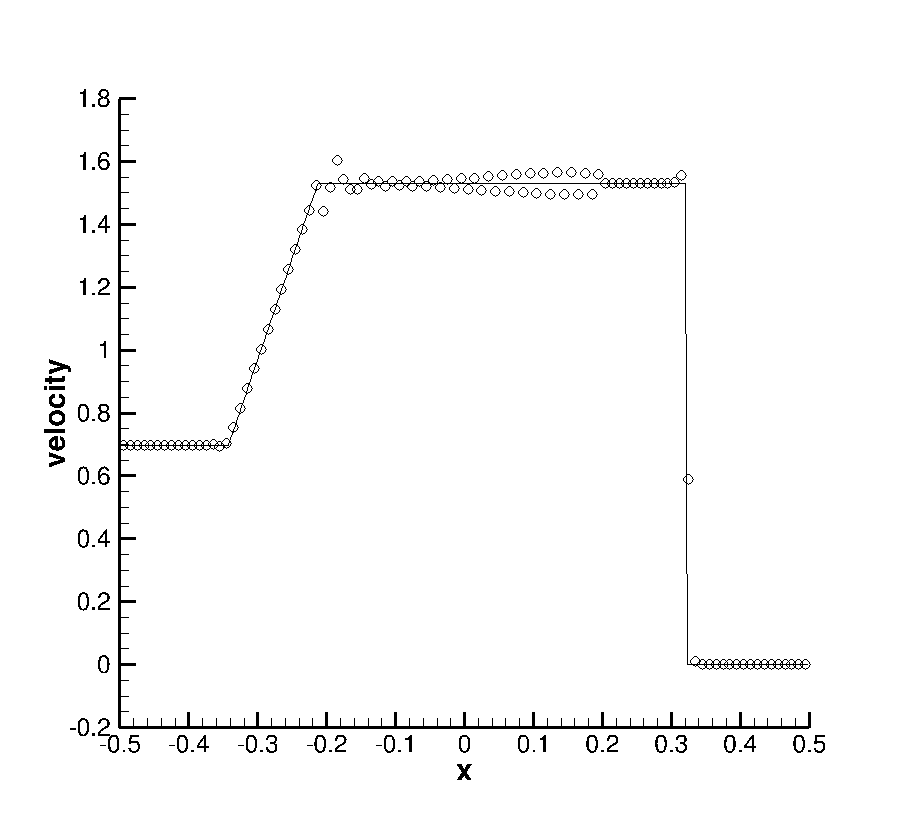}\hspace{-0.6cm}
    \includegraphics[width=4.6cm]{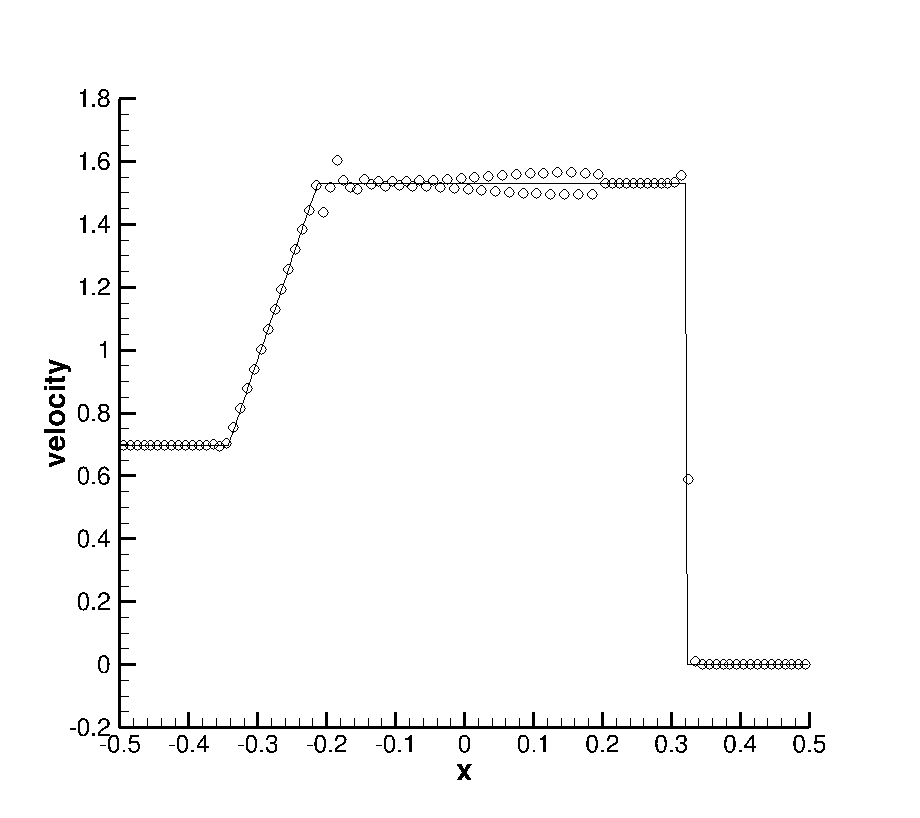} \\
    \subfloat[POS]{\includegraphics[width=4.6cm]{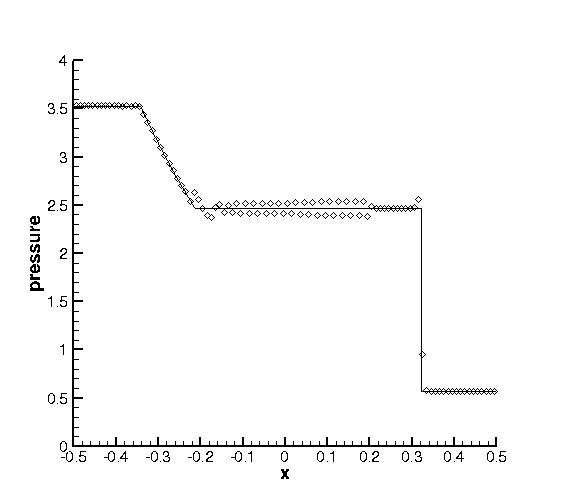}}\hspace{-0.6cm}
    \subfloat[IDP]{\includegraphics[width=4.6cm]{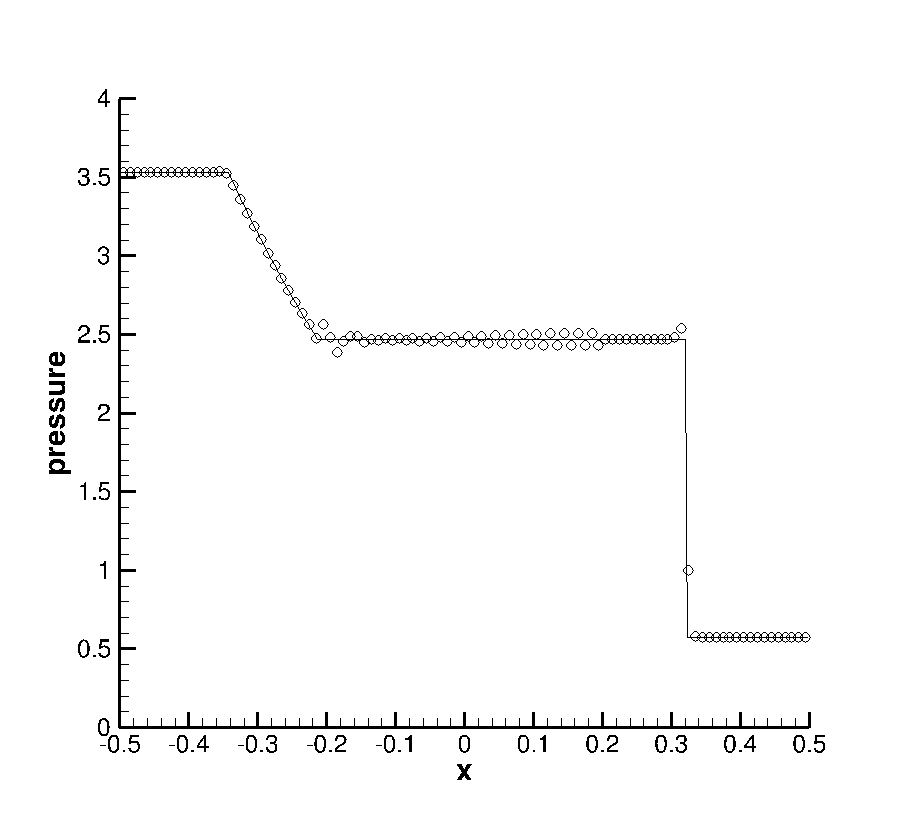}}\hspace{-0.6cm}
    \subfloat[IDPloc]{\includegraphics[width=4.6cm]{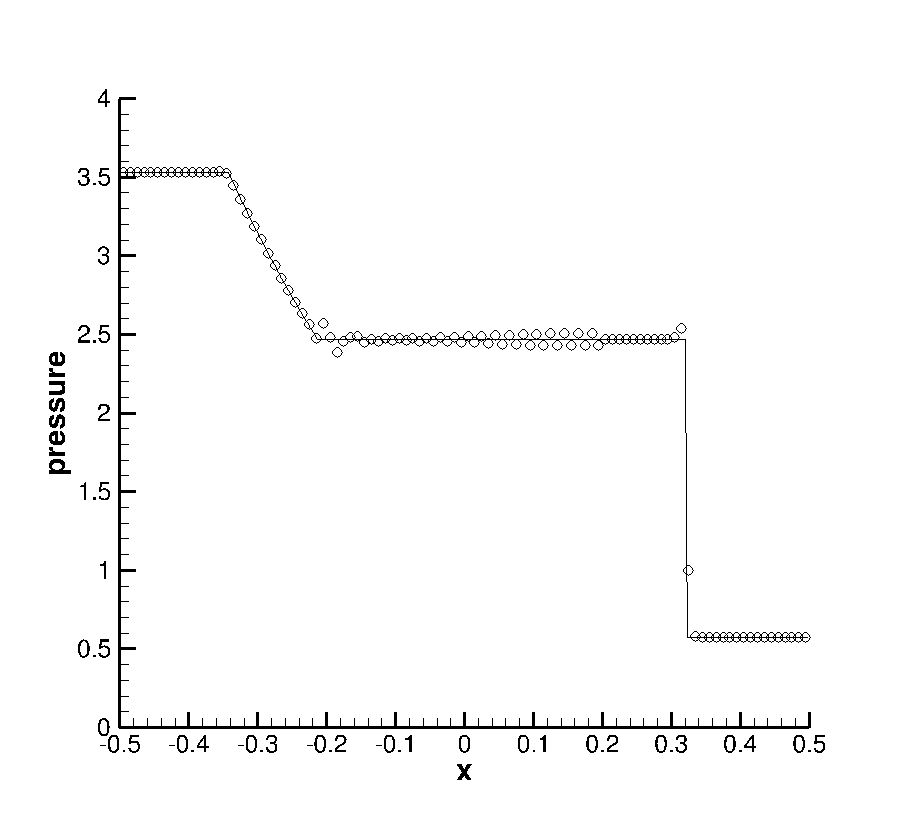}}
    \caption{DGSEM computations of the Lax problem at $t=0.2$ with $p=3$ and $N=100$ elements for the density (top), velocity (middle), and pressure (bottom).}
    \label{fig:lax}
\end{figure}



\begin{figure}
    \centering
    \includegraphics[width=4.6cm]{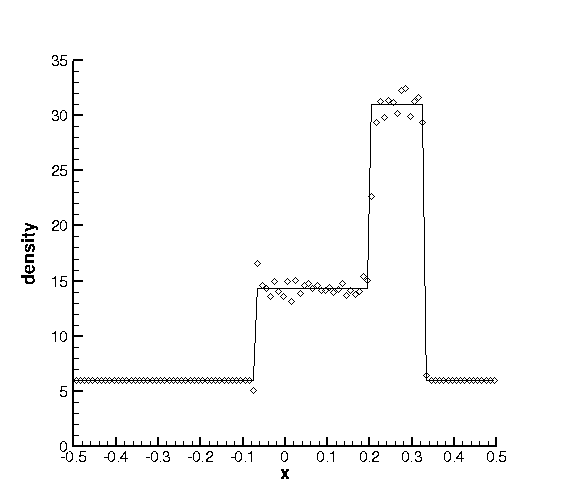}\hspace{-0.6cm}
    \includegraphics[width=4.6cm]{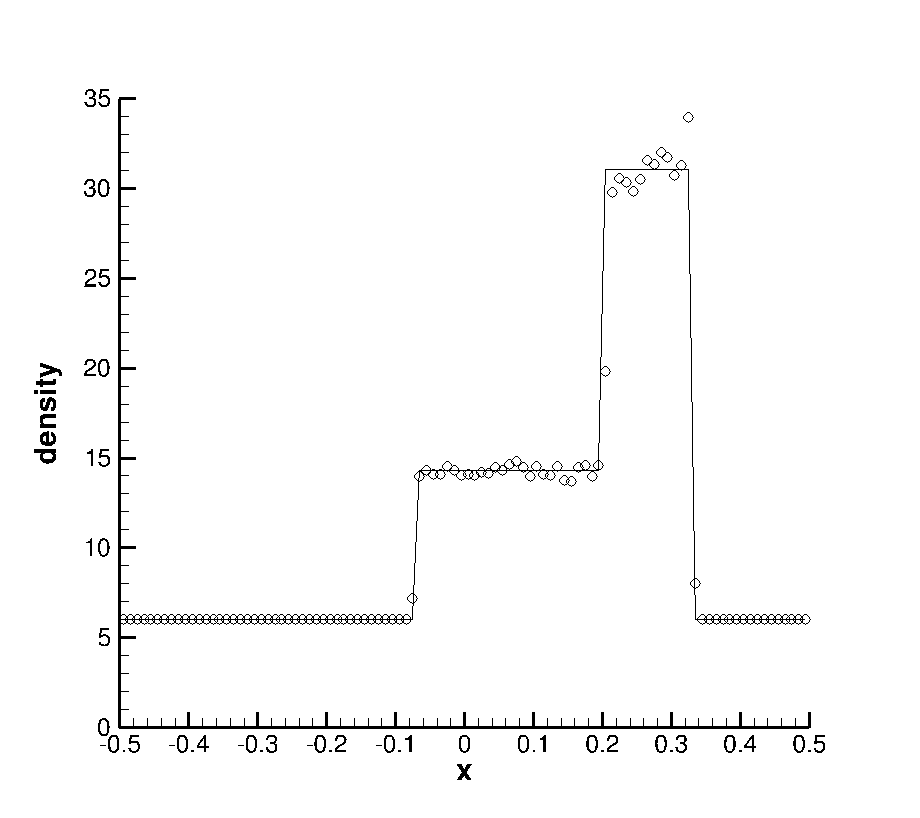}\hspace{-0.6cm}
    \includegraphics[width=4.6cm]{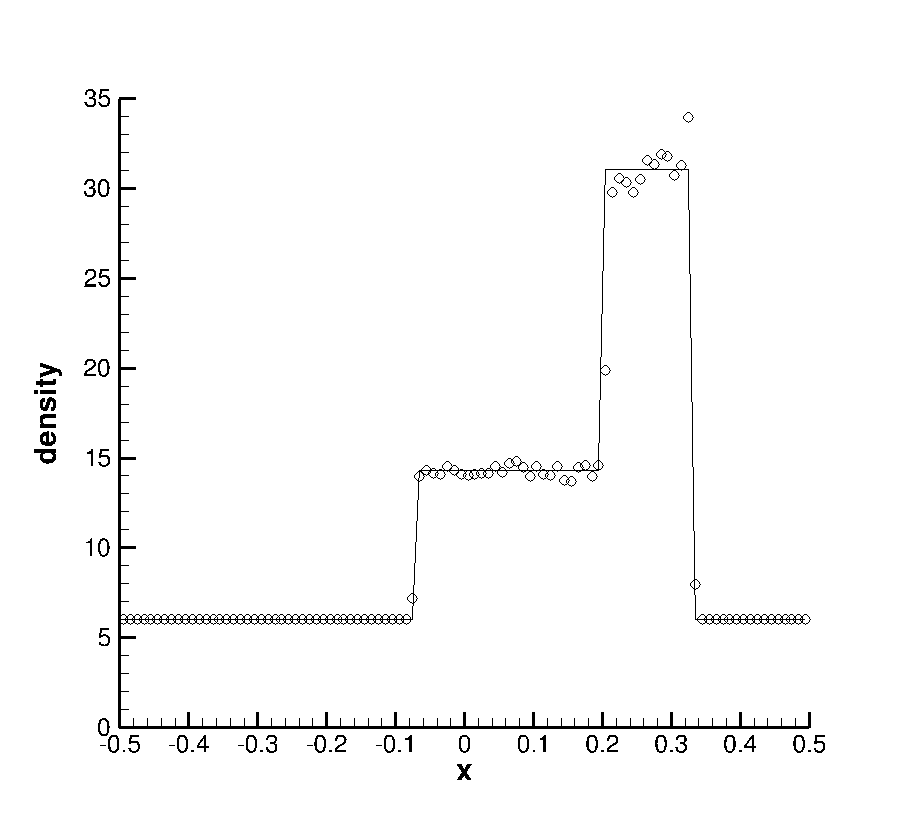} \\
    \includegraphics[width=4.6cm]{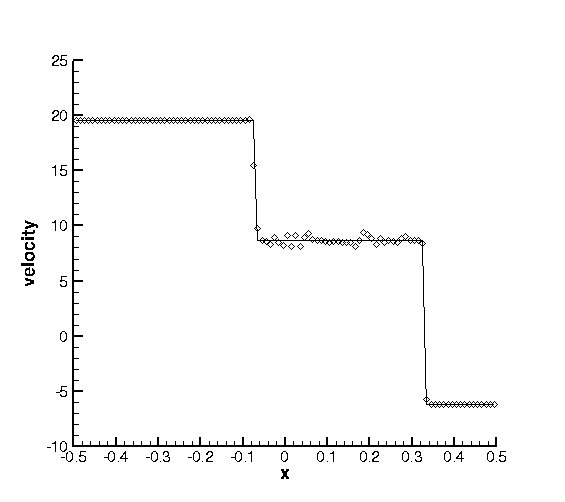}\hspace{-0.6cm}
    \includegraphics[width=4.6cm]{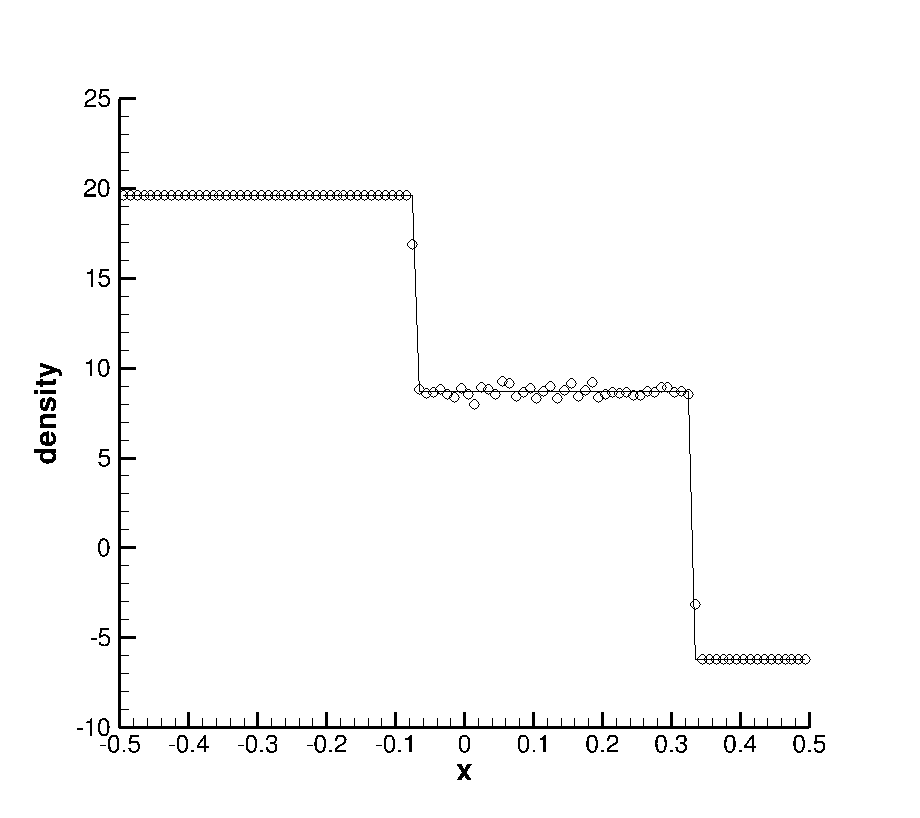}\hspace{-0.6cm}
    \includegraphics[width=4.6cm]{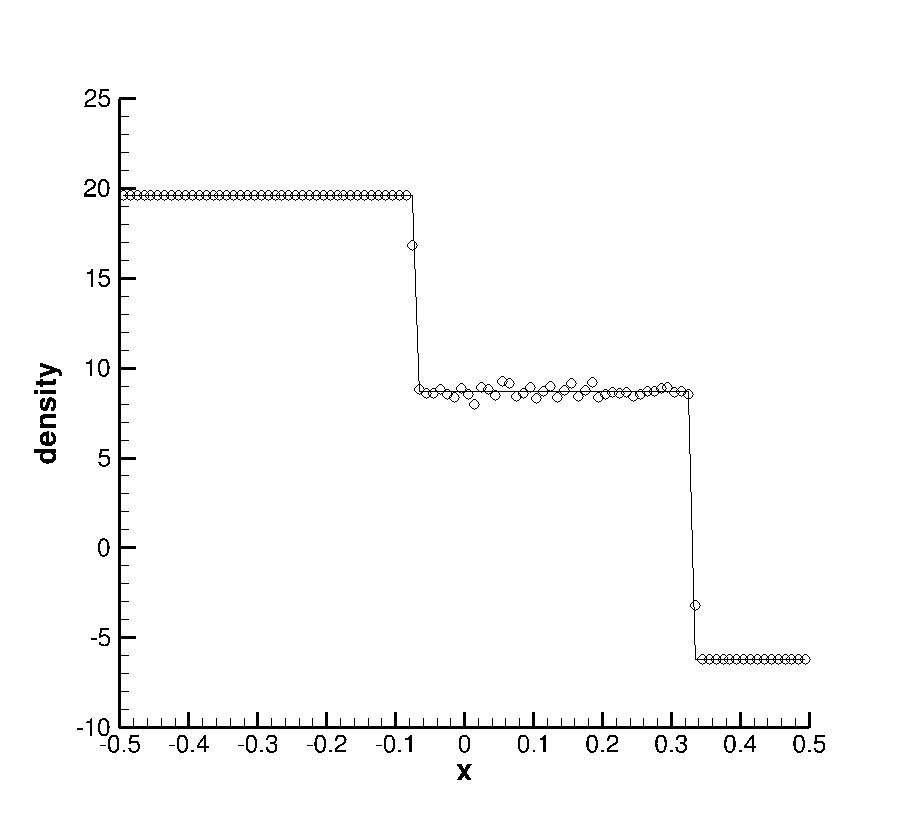} \\
    \subfloat[POS]{\includegraphics[width=4.6cm]{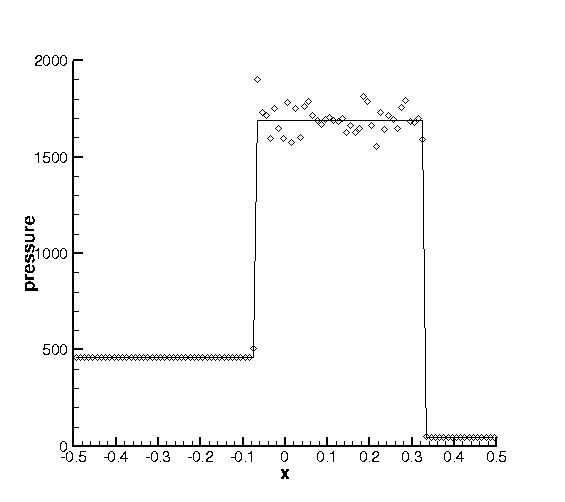}}\hspace{-0.6cm}
    \subfloat[IDP]{\includegraphics[width=4.6cm]{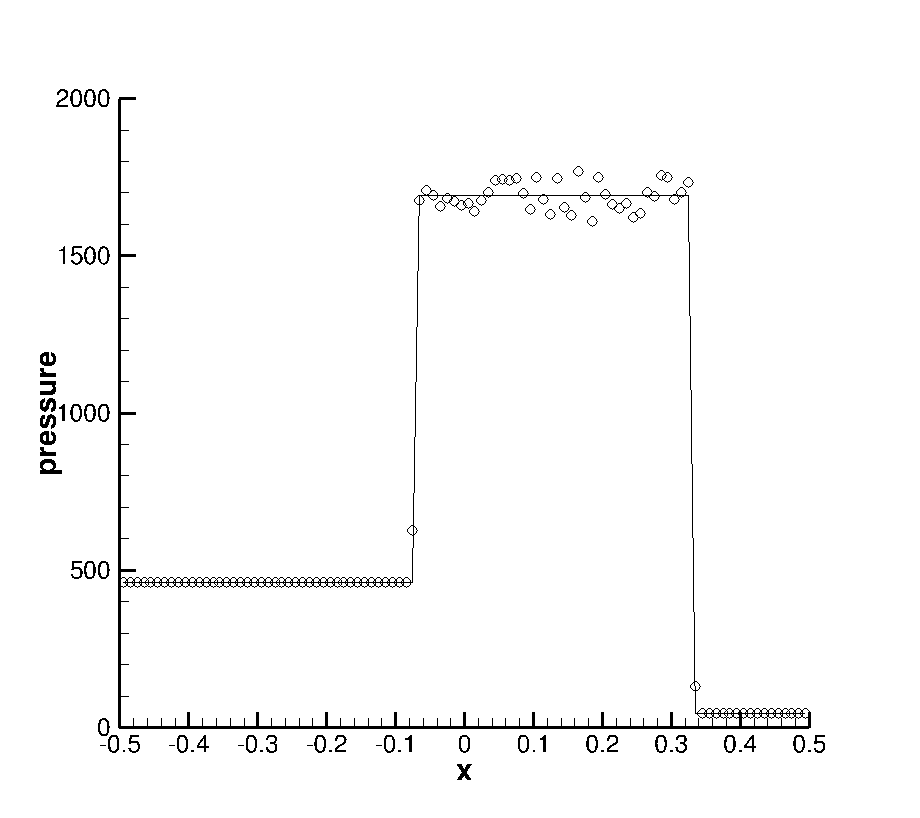}}\hspace{-0.6cm}
    \subfloat[IDPloc]{\includegraphics[width=4.6cm]{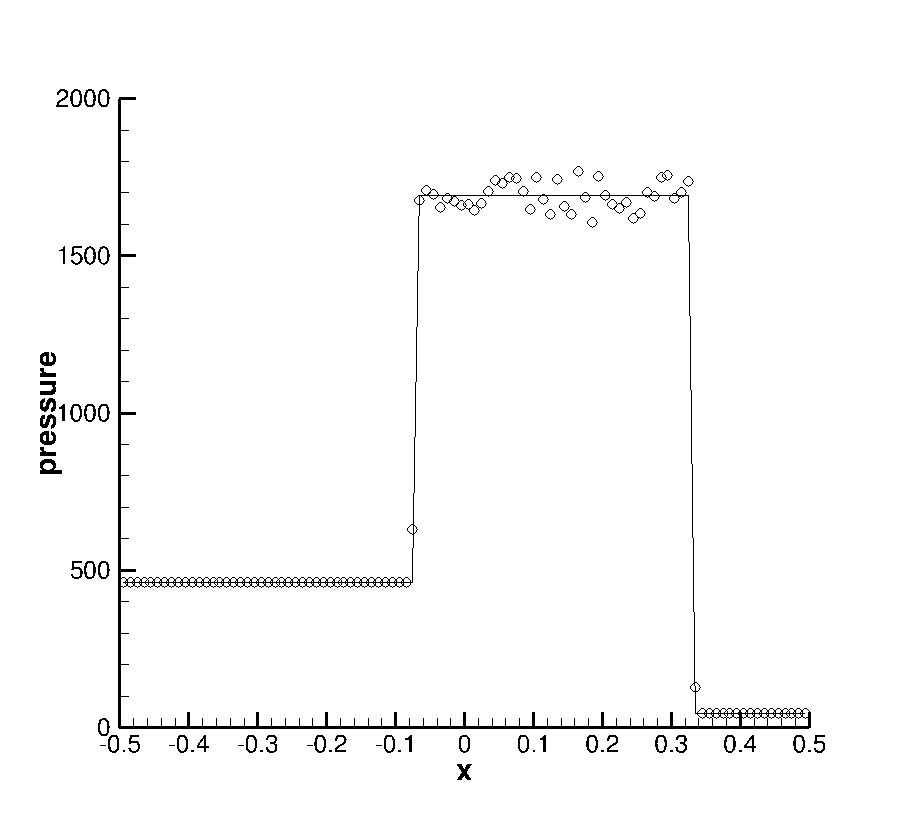}}
    \caption{DGSEM computations of the Toro 4 problem at $t=0.2$ with $p=3$ and $N=100$ elements for the density (top), velocity (middle), and pressure (bottom).}
    \label{fig:toro}
\end{figure}



%
We now reproduce the same numerical experiments but using the classical modal DG (see \cref{sec:modalDG}) with the IDP strategy. \Cref{fig:mod_DG_RP} presents the results for the threes Riemann problems and we observe a good resolution of the waves with only lower amplitude oscillations compared to the DGSEM results. Imposing the IDP property at all quadrature points may result in stronger limiting of the solution with the modal DG scheme as these are more quadrature points (compare \cref{fig:stencil_2D,fig:stencil_2D_DGSEM}). Again, algorithm \cref{eqn:modifseq} converges very fast with a global average between of $1.11$ and $1.18$ iterations evaluated over the whole computations. 
Finally note that all the computations (with either DGSEM, or modal DG scheme) require to apply one of the limiting strategy to avoid non-physical solution. 

\begin{figure}
    \centering
    \includegraphics[width=4.6cm]{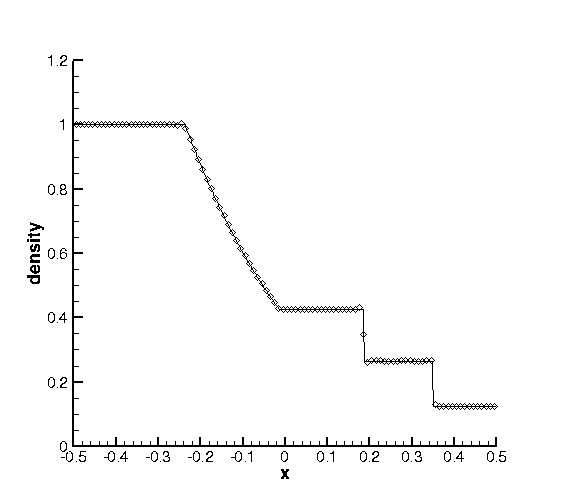}\hspace{-0.6cm}
    \includegraphics[width=4.6cm]{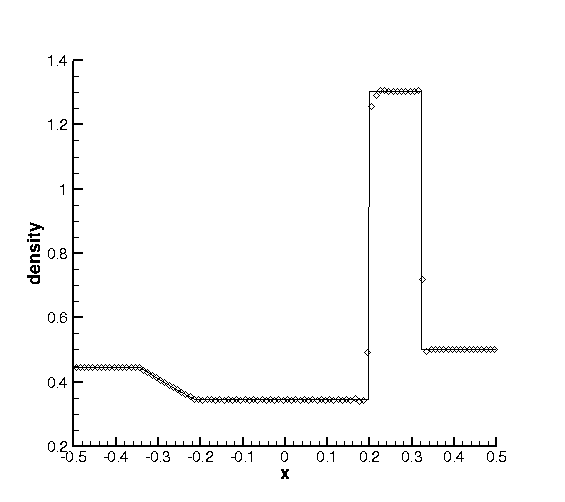}\hspace{-0.6cm}
    \includegraphics[width=4.6cm]{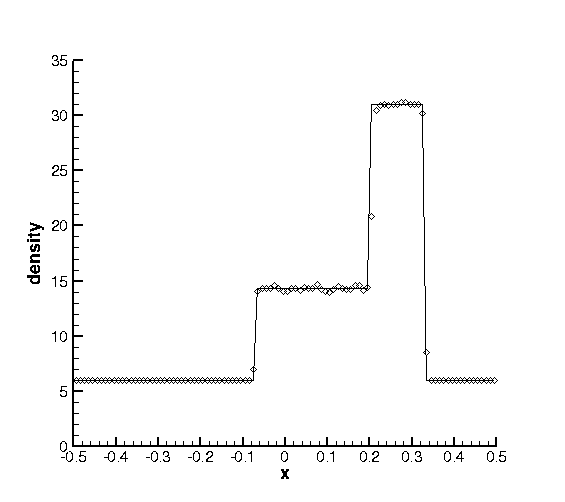} \\
    \includegraphics[width=4.6cm]{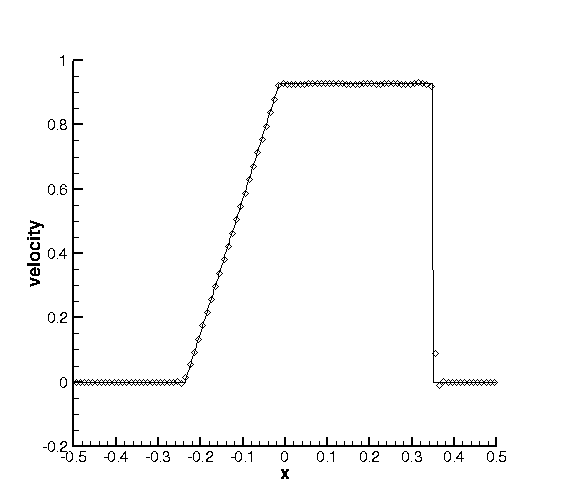}\hspace{-0.6cm}
    \includegraphics[width=4.6cm]{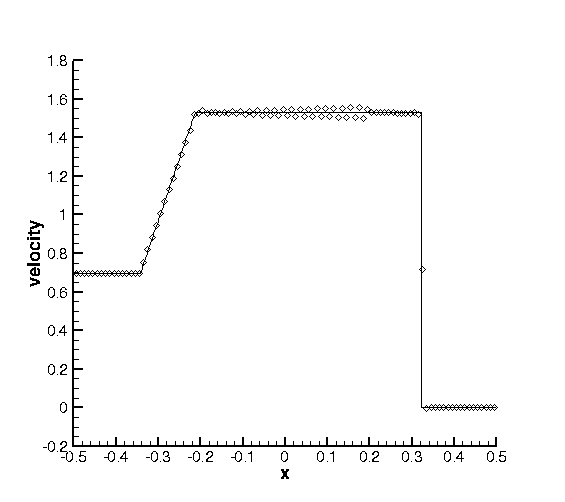}\hspace{-0.6cm}
    \includegraphics[width=4.6cm]{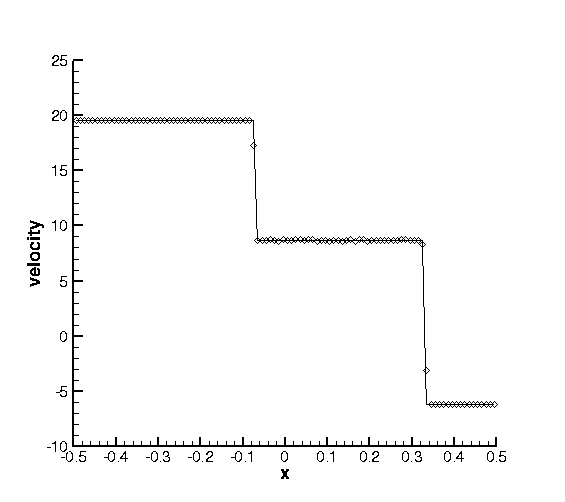} \\
    \subfloat[Sod]{\includegraphics[width=4.6cm]{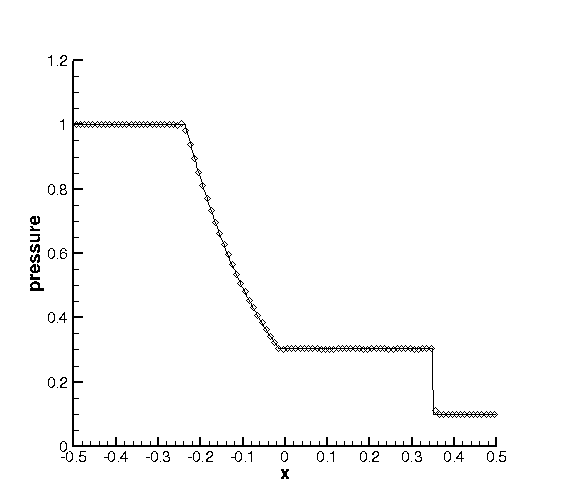}}\hspace{-0.6cm}
    \subfloat[Lax]{\includegraphics[width=4.6cm]{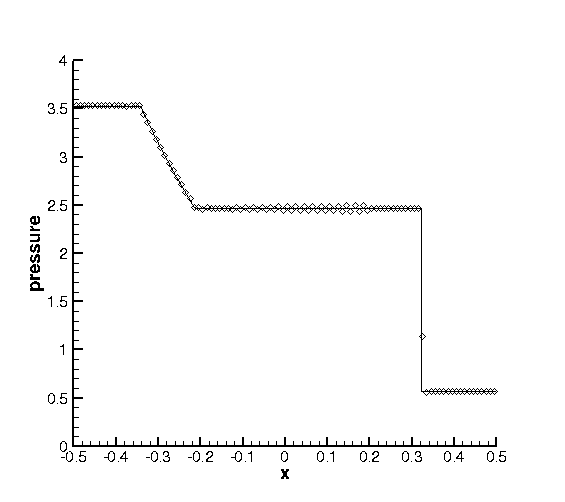}}\hspace{-0.6cm}
    \subfloat[Toro 4]{\includegraphics[width=4.6cm]{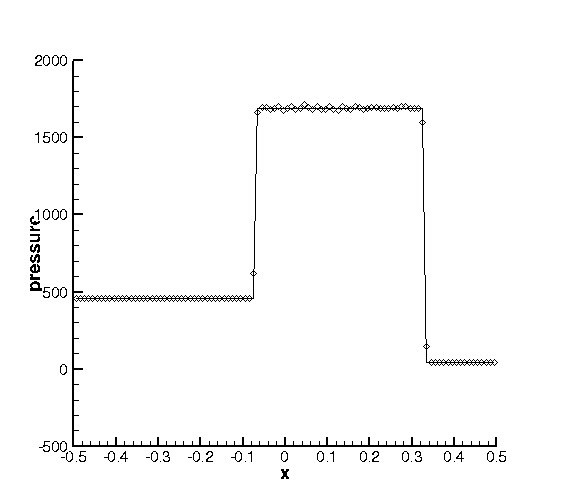}}
		\caption{Modal DG  (IDP) computations of Riemann problems with $p=3$ and $N=100$ elements for the density (top), velocity (middle), and pressure (bottom).}
    \label{fig:mod_DG_RP}
\end{figure}

\subsection{Double Mach Reflection problem}


We now consider the two-dimensional problem of a Mach $10$ shock reflection over a $30^\circ$ wedge \cite{woodward_collela_84}. Ahead of the shock, the gaz is at rest and has a density of $1.4$ and pressure of $1$. Inflow and outflow conditions are applied at the left and bottom boundaries, while a symmetry condition is applied at the top boundary. Initially, the shock is located at $x=0$ corresponding to the beginning of the wedge. We use an unstructured mesh with $132800$ quadrangles to solve the horizontally moving shock interacting with the inclined wall where slip conditions are applied. In \cref{fig:DMR_far} we can observe qualitatively similar results for the three tests. The limiting strategy induce some spurious oscillations, but the computations proved to be robust. Here again, algorithm \cref{eqn:modifseq} proves to be cheap in term of iterations to converge with global averages of $1.18$ for both the DGSEM (IDP) and modal DG (IDP) computations.


\begin{figure}
    \centering
    \subfloat[POS (DGSEM)]{\includegraphics[width=6cm]{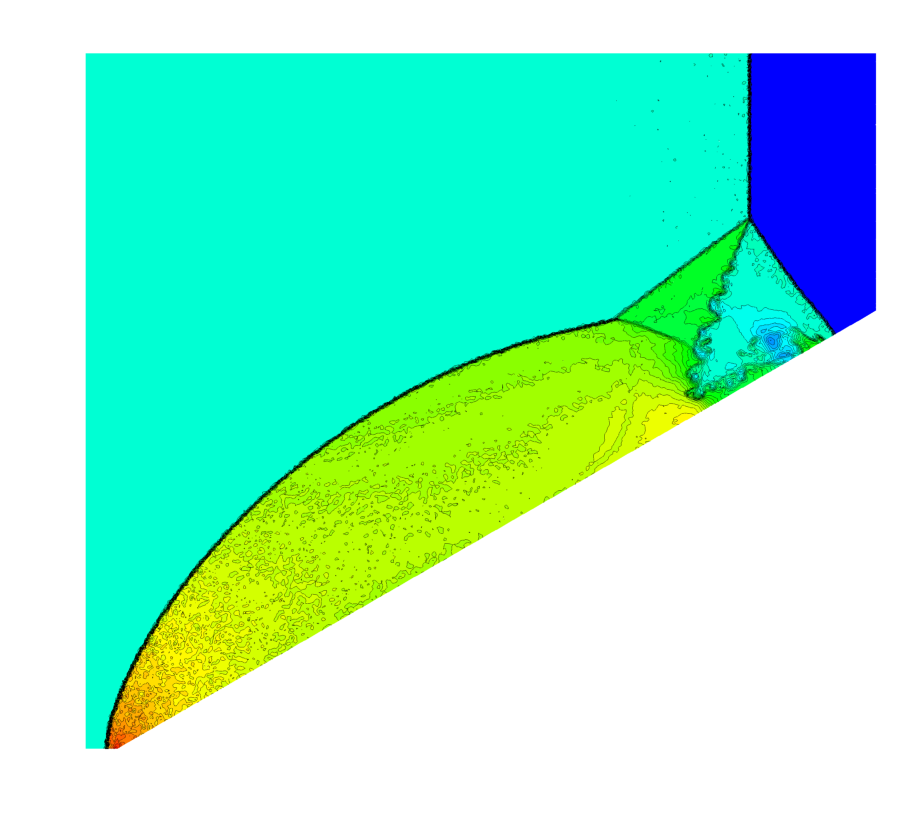}}
    \subfloat[IDP (DGSEM)]{\includegraphics[width=6cm]{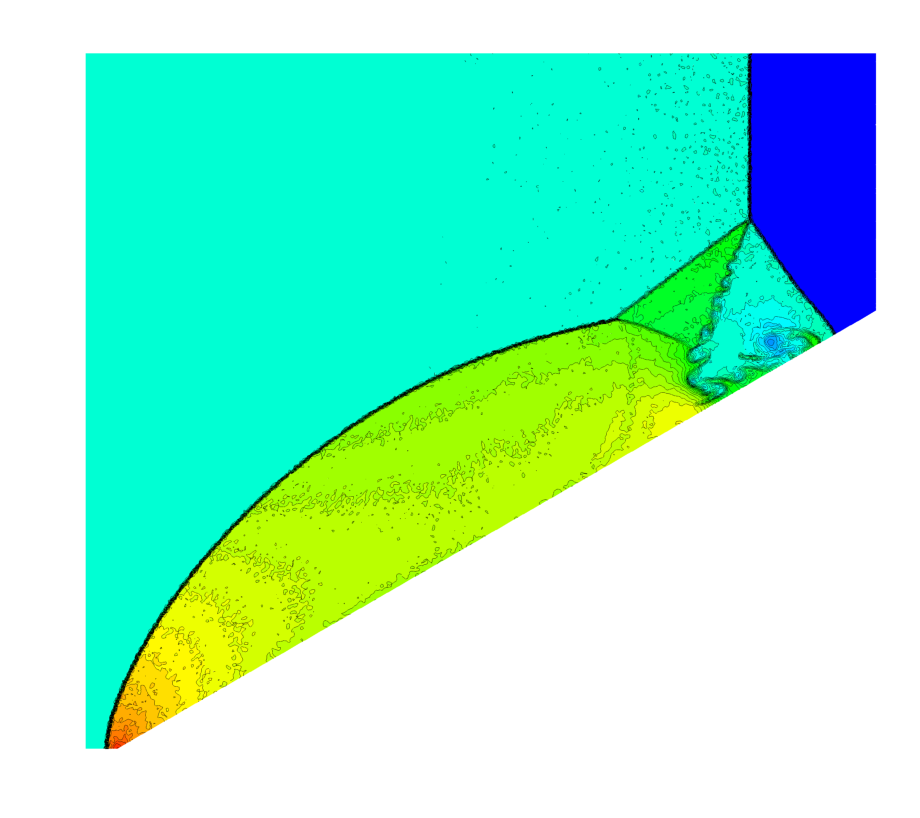}}\\
    \subfloat[IDPloc (DGSEM)]{\includegraphics[width=6cm]{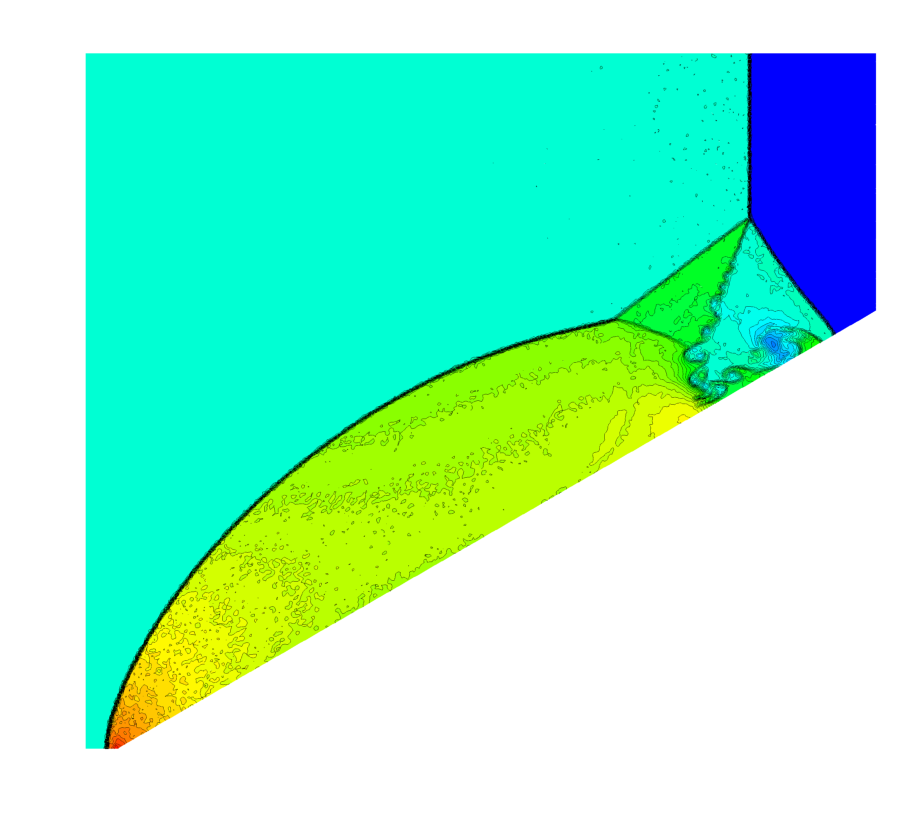}}
		\subfloat[IDP (modal DG)]{\includegraphics[width=6cm]{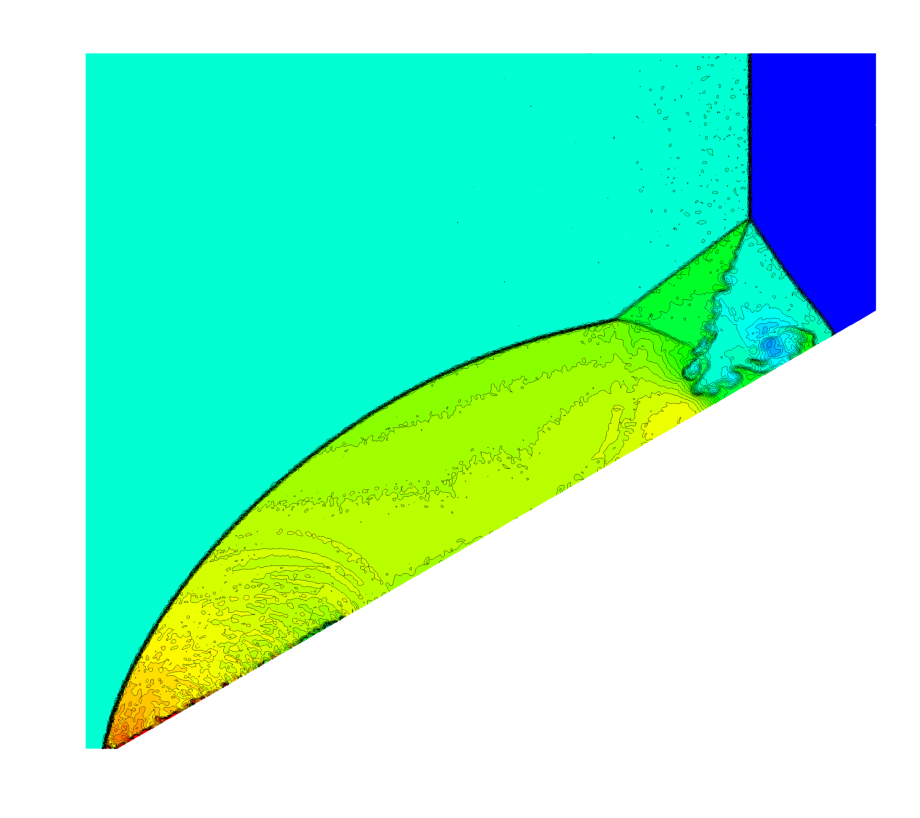}}
    \caption{Comparison of the results for the double Mach reflection problem: 39 equispaced density contours between 2.1 and 22.1.}
    \label{fig:DMR_far}
\end{figure}



%
%
\section{Conclusions}\label{sec:conclusions}

We here investigate robustness and stability properties of discretely conservative high-order spectral discontinuous methods with explicit time stepping for the approximation of hyperbolic systems of conservation laws. We derive a condition on the time step to guaranty that the cell-averaged approximate solution is a convex combination of DOFs at preceding time step and updates of invariant domain preserving and entropy stable three-point schemes. As a consequence, the cell-averaged solution lies in some convex invariant domain and we apply a posteriori scaling limiting techniques \cite{zhang2010positivity} that impose to all the DOFs to satisfy the same invariant domain properties.

The condition on the time step is evaluated from the traces of the solution at faces of the mesh and can be easily evaluated on the fly. Provided the scheme satisfies the discrete metric identities, the condition is fairly general and holds for general unstructured grid with possibly curved elements. It relies on the existence of a so-called pseudo-equilibrium state which satisfies a flux balance over each mesh element, and the existence of a quadrature rule including the traces to evaluate the cell-averaged solution. We here prove their existence in the general case and provide an iterative algorithm to evaluate the pseudo-equilibrium state. We illustrate these results with the classical modal discontinuous Galerkin and DGSEM schemes. Numerical experiments in one and two space dimensions are provided to illustrate the robustness and stability of the present approach. The extension of this framework to parabolic systems of conservation laws, e.g., the compressible Navier-Stokes, is a possible direction of future research.

\bibliographystyle{siamplain}
\bibliography{../../BIBLIO_BIB/biblio_generale}
\end{document}